\documentclass[11pt, a4paper]{article}

\usepackage[utf8]{inputenc}
\usepackage[T1]{fontenc}
\usepackage{lmodern}
\usepackage[english]{babel}
\usepackage{amsmath, amssymb, amsthm, mathtools}
\usepackage[margin=1in]{geometry}
\usepackage[numbers, sort&compress]{natbib}
\usepackage{enumitem}
\usepackage{graphicx}
\usepackage{xcolor}
\usepackage{hyperref}
\hypersetup{colorlinks=true, linkcolor=blue!60!black, citecolor=blue!50!black, urlcolor=blue!60!black}

\theoremstyle{plain}
\newtheorem{theorem}{Theorem}[section]
\newtheorem{proposition}[theorem]{Proposition}
\newtheorem{lemma}[theorem]{Lemma}
\newtheorem{corollary}[theorem]{Corollary}
\theoremstyle{definition}
\newtheorem{definition}[theorem]{Definition}
\newtheorem{assumption}[theorem]{Assumption}
\theoremstyle{remark}
\newtheorem{remark}[theorem]{Remark}

\numberwithin{equation}{section}

\newcommand{\R}{\mathbb{R}}
\newcommand{\E}{\mathbb{E}}
\newcommand{\PP}{\mathbb{P}}
\newcommand{\D}{\mathbb{D}}
\newcommand{\fH}{\mathfrak{H}}
\newcommand{\eps}{\varepsilon}

\begin{document}

\title{A CIR-Type Diffusion Driven by Hermite Processes:\\
Well-Posedness, Positivity and Malliavin Analysis}

\author{Atef Lechiheb\thanks{Toulouse School of Economics, Universit\'e Toulouse Capitole, Toulouse, France.
E-mail: \texttt{atef.lechiheb@tse-fr.eu}. ORCID: 0000-0002-9896-3408.}}

\date{}

\maketitle

\begin{abstract}
We study the stochastic differential equation
\[
dX_t=a\bigl(b(t)-X_t\bigr)\,dt+\bigl(\sigma_0+\sigma_1\sqrt{\phi_\eps(X_t)}\bigr)\,dZ_t^{(q,H)},
\qquad X_0=x_0,
\]
where $Z^{(q,H)}$ is a Hermite process of order $q\geq1$ and Hurst
parameter $H\in(1/2,1)$, $b$ is a periodic target, $\sigma_0>0$, and
$\phi_\eps$ is a smooth positive function equal to $x+\eps$ on
$[0,\infty)$. The equation is understood pathwise, in the sense of Young.
We prove that it has a unique solution. We give lower bounds on the
probability that the solution stays positive on $[0,T]$, for every
$\sigma_1\geq0$: after a Lamperti-type change of variables, which makes
the noise additive, a last-passage argument shows that the probability of
reaching $0$ before $T$ is at most of order
$\exp(-c\,\Lambda(m_0)^{2/q}/T^{2H/q})$, where $m_0$ is the smaller of the
initial level and the minimum of the target and $\Lambda$ is the natural
scale of the diffusion coefficient. For $q\le2$ and $\sigma_1=0$ we show
that the solution reaches negative values with positive probability, so
that no almost-sure positivity holds in that case. Finally, the solution
at each time belongs to $\D^{1,\infty}$, with an explicit representation
of its Malliavin derivative, and its law is absolutely continuous with
respect to Lebesgue measure. The proofs of these two results are
self-contained and rely on the additive form of the equation. Numerical
experiments based on a wavelet-type simulation of Hermite processes
illustrate the results.
\end{abstract}

\noindent\textbf{Keywords:} Hermite process; Rosenblatt process; CIR model;
long-range dependence; Young integral; positivity; Malliavin calculus;
absolute continuity.

\smallskip
\noindent\textbf{MSC 2020:} 60H10; 60G22; 60H07; 91G30.

\section{Introduction}
\label{sec:intro}

The Cox--Ingersoll--Ross (CIR) model~\cite{cox1985a}
\begin{equation}
\label{eq:classical_cir}
dX_t=a(b-X_t)\,dt+\sigma\sqrt{X_t}\,dB_t,\qquad X_0=x_0>0,
\end{equation}
where $B$ is a Brownian motion and $a,b,\sigma>0$, is one of the basic
models for the short rate. Under the Feller condition $2ab\geq\sigma^2$
the solution never reaches $0$, and its affine structure yields explicit
bond prices. Two features of interest-rate data are not captured by
models driven by a Brownian motion.

The first is long-range dependence. The autocorrelations of short rates
decay slowly, in a way that is well described by fractionally integrated
models~\cite{backus1993long,shea1991uncertainty}; long memory is also
documented in volatility~\cite{bollerslev1996modeling} and has led to
continuous-time financial models with memory~\cite{anh2005financial,AbiJaber2019}.
The second is non-Gaussianity: changes in short rates are skewed and
heavy-tailed, and standard one-factor diffusion models are rejected by
the data~\cite{ait1996,andersen1997}. A large part of this
non-Gaussianity has been attributed to jumps~\cite{das2002surprise,johannes2004};
heavy tails are also a general feature of asset
returns~\cite{cont2001empirical}.

Fractional Brownian motion (fBm) with Hurst parameter $H>1/2$ produces
long memory, and fractional CIR-type processes have been proposed for
interest rates and volatility~\cite{mishura2017,MYT2019,mishura2020};
their innovations are Gaussian. Hermite processes offer a way to combine
long memory with non-Gaussian innovations. By the non-central limit
theorem~\cite{taqqu1975,dobrushin1979non,taqqu1979convergence}, the
renormalised partial sums of a nonlinear functional of a long-memory
Gaussian sequence converge, when the memory is strong enough relative to
the Hermite rank $q$ of the functional, to a Hermite process of order $q$;
fBm appears only for $q=1$. Hermite processes of all orders share the
covariance of fBm, so that $H$ governs the memory and $q$ the departure
from Gaussianity. This is a structural motivation rather than an
empirical one, and a self-similar driver has the limitation, discussed
in Section~\ref{sec:model}, that its marginal laws do not become closer to
Gaussian under aggregation. We study
\begin{equation}
\label{eq:main_sde_intro}
dX_t=a\bigl(b(t)-X_t\bigr)\,dt+\tilde\sigma(X_t)\,dZ_t^{(q,H)},
\qquad
\tilde\sigma(x)=\sigma_0+\sigma_1\sqrt{\phi_\eps(x)},
\qquad X_0=x_0,
\end{equation}
where $Z^{(q,H)}$ is a Hermite process with $H\in(1/2,1)$,
$b(t)=b_0+b_1\sin(2\pi t/T_{\mathrm{per}})$ is a periodic target,
$\sigma_0>0$, $\sigma_1\geq0$, and $\phi_\eps$ is a smooth positive
function equal to $x+\eps$ on $[0,\infty)$ (Section~\ref{sec:model}). The
diffusion coefficient interpolates between the Vasicek case
$\sigma_1=0$ and a square-root coefficient; we call the model of CIR
type because of this square-root part and of the mean-reverting drift,
although it is not affine. The base level $\sigma_0>0$ keeps the
diffusion coefficient bounded away from $0$, which is used throughout
the paper.

For $H>1/2$, a Hermite process is not a semimartingale (already for
$q=1$), so It\^o calculus is not available. Its paths are H\"older
continuous of every order smaller than $H$, and~\eqref{eq:main_sde_intro}
can be solved pathwise with the Young integral~\cite{young1936,MR1893308}.
For $\sigma_1=0$ the equation is a Vasicek, or Ornstein--Uhlenbeck,
equation driven by a Hermite process, a model studied
in~\cite{MaejimaTudor2007,nourdin2019statistical,AssaadTudor2020,ArayaTorresTudor2024,DiezTudor2026},
mostly from the point of view of statistical inference.

\paragraph{Main results.}
\begin{enumerate}[label=(\arabic*), itemsep=2pt]
\item \emph{Well-posedness} (Theorem~\ref{thm:existence}). Equation
\eqref{eq:main_sde_intro} has a unique pathwise solution. This is an
application of the deterministic theorem of Nualart and
R\u{a}\c{s}canu~\cite{MR1893308}.
\item \emph{An equation with additive noise} (Lemma~\ref{lem:lamperti}).
The change of variables $Y=\Lambda(X)$, $\Lambda'=1/\tilde\sigma$, in the
spirit of Lamperti, Doss and Sussmann~\cite{Lamperti1964,Doss1977,Sussmann1978},
turns~\eqref{eq:main_sde_intro} into an ordinary differential equation
driven additively by $Z^{(q,H)}$, with a globally Lipschitz drift
(Lemma~\ref{lem:drift_bounds}). This is the tool behind all the other
results.
\item \emph{Positivity} (Theorems~\ref{thm:positivity}
and~\ref{thm:positivity_general}, Proposition~\ref{prop:no_as}). If
$b_0-b_1>0$, the probability that $X$ stays positive on $[0,T]$ is bounded
below by a quantity which tends to $1$ as $x_0$ and $b_0-b_1$ become large
relative to the volatility. A last-passage argument, adapted from
Mishura and Yurchenko-Tytarenko~\cite{mishura2017}, shows that $X$ can
reach $0$ only if the oscillation of $Z^{(q,H)}$ on $[0,T]$ exceeds
$\Lambda(m_0)$, with $m_0=\min(x_0,b_0-b_1)$. Combined with the
hypercontractivity of the supremum of $Z^{(q,H)}$
(Lemma~\ref{lem:moment_growth}), this gives an upper bound of order
$\exp(-c\,(\Lambda(m_0)/T^H)^{2/q})$ on the probability of reaching $0$,
in which the order of the chaos appears through the exponent $2/q$. The
bounds are informative only when $\Lambda(m_0)$ is large compared with
$T^H$, and the bound of Theorem~\ref{thm:positivity_general} does not
involve the speed $a$ (Remark~\ref{rem:feller_ratio}).
For $q\in\{1,2\}$ and $\sigma_1=0$, we show that the solution takes
negative values with positive probability at every time $t>0$.
\item \emph{Malliavin differentiability} (Theorem~\ref{thm:malliavin}).
$X_t\in\D^{1,\infty}$ for every $t$, with an explicit expression of
$DX_t$ in terms of $DZ^{(q,H)}$. The general theory of Loosveldt, Nachit
and Nourdin~\cite{loosveldt2025} requires bounded coefficients and does
not apply here; we give a direct proof, based on the differentiability of
the solution map of the additive equation.
\item \emph{Absolute continuity} (Theorem~\ref{thm:density}). For every
$t>0$, the law of $X_t$ has a density. The proof combines the
Bouleau--Hirsch criterion with a localisation of $DX_t$ near $t$, as
in~\cite{loosveldt2025}.
\end{enumerate}
For $q=1$, results (1), (4) and (5) are covered by the theory of
equations driven by fBm with
$H>1/2$~\cite{MR1893308,nualart2009malliavin,HuNualart2007,BaudoinHairer2007};
the new cases are $q\geq2$, for which the tails of $Z^{(q,H)}$ are not
sub-Gaussian, and generic a priori bounds for Young equations do not
give moments.

For the pure square-root model ($\sigma_0=0$, $\eps=0$, constant
target), the situation is different: almost-sure positivity holds for
every Hermite driver, by the pathwise argument of Mishura and
Yurchenko-Tytarenko~\cite{mishura2017}, extended to arbitrary H\"older
drivers by Di~Nunno, Mishura and Yurchenko-Tytarenko~\cite{dinunno2023}
(Remark~\ref{rem:no_as}). Our model differs from theirs by the base
volatility and the regularisation, which make the diffusion coefficient
positive at $0$; the mechanism that prevents the square-root model from
reaching $0$ is then absent.

\paragraph{Related literature.}
Mishura and Yurchenko-Tytarenko~\cite{mishura2017} define the fractional
CIR process as $X=Y^2$, where $Y$ solves an equation with additive fBm
noise and a drift that is singular at $0$, and show by a pathwise argument
that $Y$ never reaches $0$ when $H>1/2$; the case $H<1/2$ is studied
in~\cite{MYT2019}, and further path properties, including the behaviour
at zero, in~\cite{MPRYT2018}. Di~Nunno, Mishura and
Yurchenko-Tytarenko~\cite{dinunno2023} extend this approach to equations
$dY_t=b(t,Y_t)\,dt+dZ_t$ driven by an arbitrary H\"older continuous
process $Z$, with a drift that is singular at a H\"older boundary, and
deduce in their Section~5 that CIR- and CEV-type equations
$dX_t=(\kappa_0-\theta X_t)\,dt+\nu X_t^{\alpha}\,dZ_t$ with
$\alpha\in[\frac12,1)$, driven by a process of H\"older order larger than
$1/2$, have a unique non-negative solution which never reaches
$0$~\cite[Section~5.1 and Remark~5.2]{dinunno2023}. Malliavin calculus for
equations driven by fBm is developed
in~\cite{nualart2009malliavin,HuNualart2007,BaudoinHairer2007}, and for
equations driven by elements of an arbitrary Wiener chaos
in~\cite{loosveldt2025}; see also~\cite{loosveldt2025absolute} for
densities of Hermite processes. Hermite processes are introduced
in~\cite{taqqu1975,dobrushin1979non,taqqu1979convergence} and surveyed
in~\cite{tudorbook}; the Rosenblatt distribution is studied numerically
in~\cite{VeilletteTaqqu2013}; wavelet-type simulation algorithms are given
in~\cite{Ayache2025,AyacheNum2025}. Rough volatility
models~\cite{gatheral2018rough,El_Euch_2019} concern the complementary
regime $H<1/2$.

\paragraph{Organisation.}
Section~\ref{sec:prelim} recalls the facts we use on Hermite processes,
Young integrals and Malliavin calculus. Section~\ref{sec:model} introduces
the model. Section~\ref{sec:moments} contains moment estimates for the
Hermite process, and Section~\ref{sec:existence} the existence and
uniqueness result and the change of variables. Positivity is studied in
Section~\ref{sec:positivity}, Malliavin differentiability and absolute
continuity in Section~\ref{sec:malliavin}.

\section{Preliminaries}
\label{sec:prelim}

\subsection{Hermite processes}
\label{subsec:hermite}

We work on the Wiener space $(\Omega,\mathcal F,\PP)$ of a two-sided
Brownian motion $B=(B_t)_{t\in\R}$, with
$\fH=L^2(\R)$. For $f\in L^2_s(\R^q)$, the space of symmetric
square-integrable functions on $\R^q$, $I_q(f)$ denotes the multiple
Wiener--It\^o integral of order $q$ with respect to $B$; we refer
to~\cite{nourdin2012normal,Nualart} for the theory.

\begin{definition}[Hermite process]
\label{def:hermite}
Let $q\geq1$ be an integer and $H\in(\frac12,1)$. Set
\[
H_0:=1+\frac{H-1}{q}\in\Bigl(1-\frac1{2q},\,1\Bigr),
\qquad
c(H,q):=\sqrt{\frac{H(2H-1)}{q!\,\mathrm B\bigl(H_0-\tfrac12,\,2-2H_0\bigr)^q}},
\]
where $\mathrm B$ is the Beta function. The Hermite process of order $q$
and self-similarity parameter $H$ is
\begin{equation}
\label{eq:hermite_def}
Z_t^{(q,H)}=I_q\bigl(L_t^{H,q}\bigr),
\qquad
L_t^{H,q}(\xi_1,\dots,\xi_q):=c(H,q)\int_0^t\prod_{j=1}^q(s-\xi_j)_+^{H_0-3/2}\,ds,
\qquad t\geq0,
\end{equation}
with $u_+^{\gamma}:=u^{\gamma}\mathbf 1_{\{u>0\}}$.
\end{definition}

For $q=1$, $Z^{(1,H)}$ is a fractional Brownian motion; for $q=2$ it is the
Rosenblatt process. We use the following properties; see~\cite{tudorbook},
in particular Proposition~2.1 and Section~2.2 there.

\begin{proposition}
\label{prop:hermite_properties}
Let $q\geq1$ and $H\in(\frac12,1)$.
\begin{enumerate}[label=(\roman*)]
\item For all $s,t\geq0$,
$\E\bigl[Z_s^{(q,H)}Z_t^{(q,H)}\bigr]=\frac12\bigl(s^{2H}+t^{2H}-|t-s|^{2H}\bigr)$.
In particular $\E[(Z_1^{(q,H)})^2]=1$.
\item $Z^{(q,H)}$ is $H$-self-similar: for every $c>0$,
$(Z_{ct}^{(q,H)})_{t\geq0}$ and $(c^HZ_t^{(q,H)})_{t\geq0}$ have the same
law.
\item $Z^{(q,H)}$ has stationary increments.
\item For every $t\geq0$ and as $k\to\infty$,
$\E\bigl[(Z_{t+1}^{(q,H)}-Z_t^{(q,H)})(Z_{t+k+1}^{(q,H)}-Z_{t+k}^{(q,H)})\bigr]\sim H(2H-1)\,k^{2H-2}$,
so the increments are long-range dependent.
\item $Z^{(q,H)}$ has a continuous modification whose paths are
$\zeta$-H\"older continuous on every compact interval, for every
$\zeta\in(0,H)$. We always use this modification.
\item For all $s,t\geq0$,
$\|L_t^{H,q}-L_s^{H,q}\|_{L^2(\R^q)}=|t-s|^H/\sqrt{q!}$.
\item If $q\geq2$, the law of $Z_t^{(q,H)}$, $t>0$, is not Gaussian, and
its fourth cumulant is positive. If $q$ is odd, the law is symmetric,
since $B\mapsto-B$ preserves the law of $B$ and maps $I_q(f)$ to
$(-1)^qI_q(f)$; for $q=2$ it is positively skewed.
\end{enumerate}
\end{proposition}

Property (vi) follows from
$\langle L_t^{H,q},L_s^{H,q}\rangle_{L^2(\R^q)}=\frac1{2\,q!}(t^{2H}+s^{2H}-|t-s|^{2H})$,
which is established in the proof of~\cite[Proposition~2.1]{tudorbook}; it
is consistent with (i) through the isometry
$\E[(Z_t^{(q,H)})^2]=q!\,\|L_t^{H,q}\|^2_{L^2(\R^q)}$. The positivity of the
fourth cumulant in (vii) holds for every non-zero element of a Wiener
chaos of order $q\geq2$~\cite{nourdin2012normal}.

We shall use the hypercontractivity of Wiener chaos: if $F$ belongs to the
$q$-th Wiener chaos and $p\geq2$, then
\begin{equation}
\label{eq:hypercontractivity}
\E\bigl[|F|^p\bigr]^{1/p}\le(p-1)^{q/2}\,\E\bigl[F^2\bigr]^{1/2};
\end{equation}
see~\cite{nourdin2012normal} and~\cite[Eq.~(11)]{loosveldt2025}. More generally, all moments of a
random variable belonging to a finite sum of Wiener chaoses are controlled
by its second moment.

\subsection{Fractional Sobolev spaces and Young integrals}
\label{subsec:young}

Fix $T>0$. For $\alpha\in(0,\frac12)$ we use the following two spaces
of~\cite{MR1893308}, denoted there by $W_0^{\alpha,\infty}$ and
$W_T^{1-\alpha,\infty}$.

\begin{definition}
\label{def:sobolev}
Let $\alpha\in(0,\frac12)$.
\begin{enumerate}[label=(\alph*)]
\item $W_1^\alpha(0,T)$ is the space of measurable functions
$f:[0,T]\to\R$ such that
\[
\|f\|_{\alpha,1}:=\sup_{t\in[0,T]}\Bigl(|f(t)|+\int_0^t\frac{|f(t)-f(s)|}{(t-s)^{\alpha+1}}\,ds\Bigr)<\infty.
\]
\item $W_2^{1-\alpha}(0,T)$ is the space of measurable functions
$g:[0,T]\to\R$ such that
\[
\|g\|_{1-\alpha,2}:=\sup_{0\le s<t\le T}\Bigl(\frac{|g(t)-g(s)|}{(t-s)^{1-\alpha}}
+\int_s^t\frac{|g(\tau)-g(s)|}{(\tau-s)^{2-\alpha}}\,d\tau\Bigr)<\infty.
\]
\end{enumerate}
\end{definition}

Writing $C^\gamma(0,T)$ for the space of $\gamma$-H\"older continuous
functions on $[0,T]$, one has
\[
C^{1-\alpha+\delta}(0,T)\subset W_2^{1-\alpha}(0,T)\subset C^{1-\alpha}(0,T)
\]
for every $\delta>0$ with $1-\alpha+\delta\le1$; see Lemma~\ref{lem:hermite_regularity} below for
the first inclusion.

For $f\in W_1^\alpha(0,T)$ and $g\in W_2^{1-\alpha}(0,T)$, Nualart and
R\u{a}\c{s}canu define $\int_0^tf\,dg$ as a generalised Stieltjes integral,
through fractional derivatives~\cite[Section~2]{MR1893308}. When $f\in
C^{\gamma_1}$ and $g\in C^{\gamma_2}$ with $\gamma_1+\gamma_2>1$, this
integral coincides with the Riemann--Stieltjes integral, i.e.\ with the
Young integral~\cite{young1936,zahle1998}, and Young's inequality holds:
there is a constant $c_{\gamma_1,\gamma_2}<\infty$ such that, for all
$0\le u<v\le T$,
\begin{equation}
\label{eq:young_loeve}
\Bigl|\int_u^vf\,dg-f(u)\bigl(g(v)-g(u)\bigr)\Bigr|
\le c_{\gamma_1,\gamma_2}\,[f]_{\gamma_1}\,[g]_{\gamma_2}\,(v-u)^{\gamma_1+\gamma_2},
\end{equation}
where $[f]_\gamma$ denotes the $\gamma$-H\"older seminorm of $f$ on
$[0,T]$.

\subsection{Malliavin calculus}
\label{subsec:malliavin_prelim}

We use the Malliavin calculus with respect to the isonormal Gaussian
process $h\mapsto\int_\R h(u)\,dB_u$ on $\fH=L^2(\R)$~\cite{Nualart}. For
$p\geq1$, $\D^{1,p}$ is the closure of smooth cylindrical random
variables for the norm $\|F\|_{1,p}^p=\E[|F|^p]+\E[\|DF\|_\fH^p]$, where $D$
is the Malliavin derivative, and $\D^{1,\infty}=\bigcap_{p\geq1}\D^{1,p}$. We
use three standard facts.
\begin{enumerate}[label=(M\arabic*)]
\item (Chain rule~\cite[Proposition~1.2.3]{Nualart}.) If
$\varphi\in C^1(\R^m)$ has bounded partial derivatives and
$F_1,\dots,F_m\in\D^{1,p}$, then $\varphi(F_1,\dots,F_m)\in\D^{1,p}$ and
\[
D\varphi(F_1,\dots,F_m)=\sum_{i=1}^m\partial_i\varphi(F_1,\dots,F_m)\,DF_i .
\]
\item (Closability.) Let $1<p<\infty$. If $F_n\in\D^{1,p}$, $F_n\to F$ in
$L^p(\Omega)$ and $\sup_n\E[\|DF_n\|_\fH^p]<\infty$, then $F\in\D^{1,p}$, and
$DF$ is the weak limit in $L^p(\Omega;\fH)$ of a subsequence of $(DF_n)$.
Indeed, since $L^p(\Omega;\fH)$ is reflexive, a subsequence $DF_{n_k}$
converges weakly to some $\eta$; by Mazur's lemma, convex combinations of
the $DF_{n_k}$ converge strongly to $\eta$, the same convex combinations
of the $F_{n_k}$ converge to $F$ in $L^p(\Omega)$, and the closedness of
$D$~\cite[Proposition~1.2.1]{Nualart} gives $F\in\D^{1,p}$ and $DF=\eta$.
\item (Bouleau and Hirsch~\cite{Bouleau1986},
see~\cite[Section~2.1.1]{Nualart}.) If $F\in\D^{1,2}$ and $\|DF\|_\fH>0$
almost surely, then the law of $F$ is absolutely continuous with respect
to Lebesgue measure on $\R$.
\end{enumerate}
Finally, a random variable belonging to a finite sum of Wiener chaoses is
either almost surely constant or has an absolutely continuous
law~\cite{Shigekawa1980}; see also~\cite{NourdinNualartPoly2013}.

For $f\in L^2_s(\R^q)$ and $r\in\R$, $D_rI_q(f)=q\,I_{q-1}(f(\cdot,r))$.
In particular $Z_t^{(q,H)}\in\D^{1,\infty}$ and
\begin{equation}
\label{eq:DZ}
D_rZ_t^{(q,H)}=q\,I_{q-1}\bigl(L_t^{H,q}(\cdot,r)\bigr),
\qquad
\E\bigl[\|DZ_t^{(q,H)}\|_\fH^2\bigr]=q^2\,(q-1)!\,\|L_t^{H,q}\|_{L^2(\R^q)}^2=q\,t^{2H}.
\end{equation}
Since $L_t^{H,q}(\xi_1,\dots,\xi_q)=0$ as soon as $\xi_j\geq t$ for some
$j$, we have
\begin{equation}
\label{eq:DZ_adapted}
D_rZ_t^{(q,H)}=0\qquad\text{for almost every }r\geq t .
\end{equation}

\begin{lemma}
\label{lem:DZ}
\begin{enumerate}[label=(\roman*)]
\item For every $p\geq2$ and $s,t\geq0$,
\[
\E\bigl[\|DZ_t^{(q,H)}-DZ_s^{(q,H)}\|_\fH^p\bigr]\le(p-1)^{p(q-1)/2}\,q^{p/2}\,|t-s|^{pH}.
\]
\item The $\fH$-valued process $(DZ_t^{(q,H)})_{t\geq0}$ has a modification
whose paths are $\theta$-H\"older continuous on $[0,T]$ for every
$\theta\in(0,H)$, with $DZ_0^{(q,H)}=0$. Its $\theta$-H\"older constant
$K_\theta^D$ on $[0,T]$ has moments of all orders, and
$\E[\sup_{t\le T}\|DZ_t^{(q,H)}\|_\fH^p]<\infty$ for every $p\geq1$. We
always use this modification; \eqref{eq:DZ_adapted} then holds for all
$t\geq0$ simultaneously, almost surely.
\end{enumerate}
\end{lemma}

\begin{proof}
(i) Let $(e_k)_{k\geq1}$ be an orthonormal basis of $\fH$ and
$G_k:=\langle DZ_t^{(q,H)}-DZ_s^{(q,H)},e_k\rangle_\fH$. Each $G_k$
belongs to the Wiener chaos of order $q-1$, and
$\sum_k\E[G_k^2]=q\,q!\,\|L_t^{H,q}-L_s^{H,q}\|^2_{L^2(\R^q)}=q\,|t-s|^{2H}$ by
Proposition~\ref{prop:hermite_properties}(vi). By Minkowski's inequality
in $L^{p/2}(\Omega)$ and~\eqref{eq:hypercontractivity},
\[
\begin{aligned}
\E\bigl[\|DZ_t^{(q,H)}-DZ_s^{(q,H)}\|_\fH^p\bigr]^{2/p}
&=\Bigl\|\sum_{k\geq1}G_k^2\Bigr\|_{L^{p/2}(\Omega)}
\le\sum_{k\geq1}\E\bigl[|G_k|^p\bigr]^{2/p}\\
&\le(p-1)^{q-1}\sum_{k\geq1}\E\bigl[G_k^2\bigr]
=(p-1)^{q-1}\,q\,|t-s|^{2H}.
\end{aligned}
\]
(ii) By (i) and Kolmogorov's continuity theorem for processes with
values in a Banach space~\cite[Theorem~I.2.1]{revuzyor1999}, applied
with $p$ large, there is a modification which is $\theta$-H\"older
continuous for every $\theta<H-1/p$, with a H\"older constant in
$L^p(\Omega)$; letting $p\to\infty$ gives the claims on $K^D_\theta$.
Since $DZ_0^{(q,H)}=0$, $\sup_{t\le T}\|DZ_t^{(q,H)}\|_\fH\le
K_\theta^DT^\theta$. Finally, \eqref{eq:DZ_adapted} holds almost surely
for each rational $t$, and extends to all $t$ by continuity.
\end{proof}

The following scaling property is~\cite[Lemma~6.5]{loosveldt2025}; we
include the short argument. For $0<\delta\le t$, let $P_\delta$ denote
the orthogonal projection of $\fH$ onto $L^2([t-\delta,t])$.

\begin{lemma}
\label{lem:scaling}
For $0<\delta\le t$, the random variable
$\delta^{-H}\|P_\delta DZ_t^{(q,H)}\|_\fH$ has the same law as
\[
N:=\Bigl(\int_0^1\bigl|D_rZ_1^{(q,H)}\bigr|^2\,dr\Bigr)^{1/2},
\]
and $\PP(N=0)=0$.
\end{lemma}

\begin{proof}
The map $B\mapsto\tilde B$,
$\tilde B_u:=\delta^{-1/2}(B_{t-\delta+\delta u}-B_{t-\delta})$, $u\in\R$,
preserves the law of the two-sided Brownian motion. A change of variables
in~\eqref{eq:hermite_def} shows that
$Z_t^{(q,H)}-Z_{t-\delta}^{(q,H)}=\delta^H\tilde Z_1$, where $\tilde Z$ is
the Hermite process built from $\tilde B$, and that, for
$r=t-\delta+\delta u$,
$D_r(Z_t^{(q,H)}-Z_{t-\delta}^{(q,H)})=\delta^{H-1/2}\tilde D_u\tilde Z_1$,
where $\tilde D$ is the Malliavin derivative with respect to $\tilde B$.
By~\eqref{eq:DZ_adapted}, $P_\delta DZ_{t-\delta}^{(q,H)}=0$, hence
\[
\|P_\delta DZ_t^{(q,H)}\|_\fH^2
=\int_{t-\delta}^t\bigl|D_r(Z_t^{(q,H)}-Z_{t-\delta}^{(q,H)})\bigr|^2dr
=\delta\cdot\delta^{2H-1}\int_0^1\bigl|\tilde D_u\tilde Z_1\bigr|^2du,
\]
which has the law of $\delta^{2H}N^2$. The random variable $N^2$ belongs
to a finite sum of Wiener chaoses (of orders $0,2,\dots,2q-2$), and
\[
\E[N^2]=q^2\,(q-1)!\int_0^1\bigl\|L_1^{H,q}(\cdot,r)\bigr\|^2_{L^2(\R^{q-1})}\,dr>0,
\]
since $L_1^{H,q}(\xi,r)>0$ when $\xi_1,\dots,\xi_{q-1}<r<1$. Hence $N^2$ is
either almost surely equal to $\E[N^2]>0$, or has an absolutely continuous
law; in both cases $\PP(N=0)=0$.
\end{proof}

\section{The model}
\label{sec:model}

\subsection{The equation}

We consider, for $t\in[0,T]$,
\begin{equation}
\label{eq:main_sde}
X_t=x_0+\int_0^ta\bigl(b(s)-X_s\bigr)\,ds+\int_0^t\tilde\sigma(X_s)\,dZ_s^{(q,H)},
\end{equation}
where the last integral is the pathwise integral of
Section~\ref{subsec:young},
\begin{equation}
\label{eq:drift}
b(t)=b_0+b_1\sin\Bigl(\frac{2\pi t}{T_{\mathrm{per}}}\Bigr),
\qquad
\tilde\sigma(x)=\sigma_0+\sigma_1\sqrt{\phi_\eps(x)},
\end{equation}
and $\phi_\eps$ is the function defined in~\eqref{eq:phi_eps} below.

\begin{assumption}
\label{ass:main}
$x_0\in\R$, $a>0$, $b_0\in\R$, $b_1\geq0$, $T_{\mathrm{per}}>0$,
$\sigma_0>0$, $\sigma_1\geq0$, $\eps>0$, $q\geq1$ is an integer and
$H\in(\frac12,1)$.
\end{assumption}

This assumption is in force throughout the paper. We write
$\|b\|_\infty=|b_0|+b_1$. The positivity results of
Section~\ref{sec:positivity} require, in addition,
\begin{equation}
\label{eq:target_positivity}
x_0>0
\qquad\text{and}\qquad
b_*:=\min_{t\in\R}b(t)=b_0-b_1>0 .
\end{equation}

The drift pulls $X$ towards the periodic target $b(t)$ at speed $a$. The
periodic form of $b$ is chosen for definiteness: the proofs only use that
$b$ is bounded and continuous, and, for positivity, that $\inf b>0$. The parameter $H$
governs the memory of the noise and $q$ its departure from Gaussianity.
By self-similarity, an increment of $Z^{(q,H)}$ over a lag $\Delta$ has
the law of $\Delta^HZ_1^{(q,H)}$, so the skewness and kurtosis of the
noise are the same at all time scales. This is a restriction of the
model: empirical return distributions tend to become closer to Gaussian
under temporal aggregation~\cite{cont2001empirical}, which a
self-similar driver cannot reproduce. On $[0,\infty)$, the volatility is
$\sigma_0+\sigma_1\sqrt{x+\eps}$: the level-dependent part
$\sigma_1\sqrt{x+\eps}$ is of CIR type, and the base level $\sigma_0>0$
keeps the volatility away from zero when rates are low. For $\sigma_1=0$
and $q=1$, \eqref{eq:main_sde} is the fractional Vasicek model
of~\cite{cheridito2003} with a periodic target; for $\sigma_1=0$ and
$q\geq2$ it is the Hermite Vasicek model of~\cite{nourdin2019statistical}
with a periodic target. The fractional CIR model of~\cite{mishura2017},
and its analogue driven by $Z^{(q,H)}$ considered
in~\cite{dinunno2023}, correspond to $\sigma_0=0$ and $\eps=0$, which are
excluded here; they are related models, not special cases.

\subsection{The regularised square root}

Let $\chi:\R\to[0,1]$ be the smooth transition function
\[
\chi(t)=0\ \ (t\le-1),
\qquad
\chi(t)=\frac{e^{-1/(t+1)^2}}{e^{-1/(t+1)^2}+e^{-1/t^2}}\ \ (-1<t<0),
\qquad
\chi(t)=1\ \ (t\geq0),
\]
and set
\begin{equation}
\label{eq:phi_eps}
\phi_\eps(x):=(x+\eps)\,\chi\Bigl(\frac x\eps\Bigr)+\frac\eps2\Bigl(1-\chi\Bigl(\frac x\eps\Bigr)\Bigr),
\qquad x\in\R .
\end{equation}

\begin{lemma}
\label{lem:phi}
The function $\phi_\eps$ belongs to $C^\infty(\R)$ and satisfies:
\begin{enumerate}[label=(\alph*)]
\item $\phi_\eps(x)=x+\eps$ for $x\geq0$ and $\phi_\eps(x)=\eps/2$ for
$x\le-\eps$;
\item $0<\phi_\eps(x)\le\eps$ for $x\le0$, and
$m_\eps:=\inf_{x\in\R}\phi_\eps(x)>0$ (the function $\phi_\eps$ is not
monotone on $(-\eps,0)$, and $m_\eps<\eps/2$);
\item all derivatives of $\phi_\eps$ of order $k\geq1$ are bounded on $\R$.
\end{enumerate}
\end{lemma}

\begin{proof}
The function $\chi$ is $C^\infty$, non-decreasing on $(-1,0)$, equal to $0$
on $(-\infty,-1]$ and to $1$ on $[0,\infty)$, and all its derivatives
vanish outside $(-1,0)$. Hence $\phi_\eps\in C^\infty(\R)$, and (a)
follows from $\chi(x/\eps)=1$ for $x\geq0$ and $\chi(x/\eps)=0$ for
$x\le-\eps$. For $x\in(-\eps,0)$, $\phi_\eps(x)$ is a convex combination of
$x+\eps\in(0,\eps)$ and $\eps/2$, so $0<\phi_\eps(x)\le\eps$. Since
$\phi_\eps$ is continuous and positive on the compact interval
$[-\eps,0]$, its minimum there is positive; together with (a), this gives
$m_\eps>0$ and (b). For (c), $\phi_\eps'=1$ and $\phi_\eps^{(k)}=0$
($k\geq2$) on $(0,\infty)$, all derivatives vanish on $(-\infty,-\eps)$,
and the derivatives are continuous, hence bounded, on $[-\eps,0]$.
\end{proof}

\begin{lemma}
\label{lem:sigma}
The function $\tilde\sigma$ belongs to $C^\infty(\R)$,
$\tilde\sigma\geq\sigma_0>0$, and all its derivatives of order $k\geq1$
are bounded on $\R$. In particular $\tilde\sigma$ and $\tilde\sigma'$ are
globally Lipschitz. Moreover $\tilde\sigma$ is non-decreasing on
$[0,\infty)$, $\tilde\sigma(x)\le\sigma_0+\sigma_1\sqrt\eps$ for $x\le0$,
and, for $x\geq0$,
\[
\tilde\sigma(x)=\sigma_0+\sigma_1\sqrt{x+\eps},
\qquad
\tilde\sigma'(x)=\frac{\sigma_1}{2\sqrt{x+\eps}},
\qquad
\tilde\sigma''(x)=-\frac{\sigma_1}{4(x+\eps)^{3/2}} .
\]
\end{lemma}

\begin{proof}
Since $\phi_\eps\geq m_\eps>0$ and the square root is smooth on
$(0,\infty)$, $\sqrt{\phi_\eps}\in C^\infty(\R)$. By the formula of
Fa\`a di Bruno, for $k\geq1$ the $k$-th derivative of
$\sqrt{\phi_\eps}$ is a finite sum of terms of the form
\[
c_j\,\phi_\eps^{\,1/2-j}\prod_{i=1}^j\phi_\eps^{(n_i)},
\qquad1\le j\le k,\quad n_i\geq1,\quad n_1+\dots+n_j=k,
\]
with constants $c_j$. Each factor $\phi_\eps^{(n_i)}$ is bounded by
Lemma~\ref{lem:phi}(c), and $\phi_\eps^{1/2-j}\le m_\eps^{1/2-j}$, so
every such term is bounded. The lower bound $\tilde\sigma\geq\sigma_0$
is clear. The last statements follow from Lemma~\ref{lem:phi}(a)--(b).
\end{proof}

Thus $\tilde\sigma$ is a smooth version of the hybrid volatility
$\sigma_0+\sigma_1\sqrt{x_++\eps}$, which is not differentiable at
$x=0$. The regularisation only matters on $(-\infty,0)$, a region
that the solution visits with small probability in the regime of the
positivity results of Section~\ref{sec:positivity}. The parameter
$\eps$ is a fixed technical constant, not a parameter to be estimated.

\section{Moment estimates for the Hermite process}
\label{sec:moments}

This section collects the properties of the paths of $Z^{(q,H)}$ that are
used below: membership in the spaces of Section~\ref{subsec:young}, and
moment bounds for the H\"older constant and for the supremum
\[
C_T:=\sup_{s\in[0,T]}\bigl|Z_s^{(q,H)}\bigr| .
\]

\begin{lemma}[Regularity of the Hermite process]
\label{lem:hermite_regularity}
For every
$\beta \in (0, H - \tfrac{1}{2})$, the sample paths of
$Z^{(q,H)}$ belong almost surely to
$W_2^{H-\beta}(0,T)$.
\end{lemma}

\begin{proof}
We first establish a moment bound on the increments of $Z^{(q,H)}$.

\noindent\emph{Step~1: Moment bound.}
By the self-similarity of $Z^{(q,H)}$ with index $H$
(Proposition~\ref{prop:hermite_properties}(ii)) and the stationarity of
its increments (Proposition~\ref{prop:hermite_properties}(iii)),
\begin{equation}
\label{eq:hermite_moment}
\E\!\left[|Z_t^{(q,H)} - Z_s^{(q,H)}|^p\right]
= |t-s|^{pH}\,\E\!\left[|Z_1^{(q,H)}|^p\right],
\quad p \geq 1,\; s,t \in [0,T].
\end{equation}
Applying~\eqref{eq:hypercontractivity} with $F = Z_1^{(q,H)}$ and using
$\E[(Z_1^{(q,H)})^2] = 1$
(Proposition~\ref{prop:hermite_properties}(i)),
we obtain $\E[|Z_1^{(q,H)}|^p] \leq (p-1)^{pq/2} < \infty$
for $p\geq2$, and for $1\le p<2$ by Lyapunov's inequality.
Hence~\eqref{eq:hermite_moment} gives, with $C_{p,q} :=
\E[|Z_1^{(q,H)}|^p] < \infty$:
\[
\E\!\left[|Z_t^{(q,H)} - Z_s^{(q,H)}|^p\right]
\leq C_{p,q}\, |t-s|^{pH}.
\]

\noindent\emph{Step~2: Kolmogorov continuity criterion.}
Fix $\beta\in(0,H-\tfrac12)$ and choose $p>1/\beta$, so that
$1/p<\beta$ (in particular $pH>1$, since $H>1/2$). The moment bound
of Step 1 reads $\E|Z_t^{(q,H)}-Z_s^{(q,H)}|^p \leq
C_{p,q}\,|t-s|^{1+(pH-1)}$ with $pH-1>0$, so Kolmogorov's
continuity theorem~\cite[Theorem~I.2.1]{revuzyor1999} shows that the
continuous modification of Proposition~\ref{prop:hermite_properties}(v)
has sample paths that are
$\gamma$-H\"{o}lder continuous on $[0,T]$ for every
$\gamma < H-1/p$. Since $1/p<\beta$, we may fix some $\gamma$ with
\[
H - \beta < \gamma < H - \tfrac1p .
\]

\noindent\emph{Step~3: Embedding into $W_2^{H-\beta}$.}
Let $K$ be the (a.s.\ finite) $\gamma$-H\"older constant of the
continuous modification, with $\gamma>H-\beta$ fixed in Step~2. For
$0\le s<t\le T$, the two terms in the norm
$\|\cdot\|_{H-\beta,2}$ of Definition~\ref{def:sobolev}(b) (with
$1-\alpha=H-\beta$) satisfy
\[
\frac{|Z_t^{(q,H)}-Z_s^{(q,H)}|}{(t-s)^{H-\beta}}
\le K\,T^{\gamma-(H-\beta)},
\]
\[
\int_s^t\frac{|Z_\tau^{(q,H)}-Z_s^{(q,H)}|}{(\tau-s)^{1+H-\beta}}\,d\tau
\le K\int_0^T r^{\gamma-(H-\beta)-1}\,dr
=\frac{K\,T^{\gamma-(H-\beta)}}{\gamma-(H-\beta)} .
\]
The last integral is finite because $\gamma>H-\beta$. Hence $\|Z^{(q,H)}\|_{H-\beta,2}<\infty$ a.s., i.e.\ the paths
belong to $W_2^{H-\beta}(0,T)$. This is the embedding
$C^{\gamma}\subset W_2^{H-\beta}$, $\gamma>H-\beta$, of
Section~\ref{subsec:young}.
\end{proof}

The H\"older constant of $Z^{(q,H)}$ has moments of all orders. This is
used to control $\E[C_T^p]$ in Section~\ref{sec:positivity}.

\begin{lemma}[Moment bound on the H\"older constant]
\label{lem:grr}
Fix $p\geq 1$ with $pH>1$, fix
$m\in(0,pH-1)$ and set $\gamma_{p,m} := m/p \in (0,H-1/p)$. There is a
constant $c_{p,m}<\infty$, depending only on $p$ and $m$ (not on $T$
or on the sample path), such that the random variable
\[
  K_{p,m} \;:=\; \sup_{0\le s<t\le T} \frac{\bigl|Z_t^{(q,H)}-Z_s^{(q,H)}\bigr|}{|t-s|^{\gamma_{p,m}}}
\]
computed for the continuous modification of $Z^{(q,H)}$, satisfies
\[
  \E\bigl[K_{p,m}^p\bigr] \;\le\;
  \frac{2\,c_{p,m}\, C_{p,q}}{(pH-m-1)(pH-m)}\; T^{\,pH-m} \;<\;\infty,
\]
where $C_{p,q}=\E\bigl[|Z_1^{(q,H)}|^p\bigr]$ is the constant
from~\eqref{eq:hermite_moment} (Step~1 of the proof of
Lemma~\ref{lem:hermite_regularity}).
\end{lemma}

\begin{proof}
We use the Garsia--Rodemich--Rumsey lemma~\cite{garsia1970} in the
form of~\cite[Lemma~A.3.1 and~(A.10)]{Nualart} (one-dimensional
parameter, $\Psi(x)=x^p$, $\varrho(x)=x^{(m+2)/p}$, where we write
$\varrho$ for the function denoted $p$ in~\cite{Nualart}): for every continuous
$f:[0,T]\to\R$, every $p\geq 1$ and every $m>0$, setting
\[
  \Gamma_{p,m}(f) \;:=\; \int_0^T\!\!\int_0^T
  \frac{|f(t)-f(s)|^p}{|t-s|^{m+2}}\,ds\,dt,
\]
one has, for every $0\le s<t\le T$,
\begin{equation}
  \label{eq:grr}
  |f(t)-f(s)|^p \;\le\; c_{p,m}\,|t-s|^{m}\, \Gamma_{p,m}(f),
\end{equation}
for a constant $c_{p,m}$ depending only on $p$ and $m$.
Applying~\eqref{eq:grr} with $f=Z^{(q,H)}$ gives
\begin{equation}
  \label{eq:Kp_bound}
  K_{p,m}^p \;\le\; c_{p,m}\, \Gamma_{p,m}\bigl(Z^{(q,H)}\bigr)
  \qquad\text{a.s.}
\end{equation}
The integrand of $\Gamma_{p,m}$ is non-negative, so by Tonelli's
theorem and the moment bound~\eqref{eq:hermite_moment},
\[
  \E\bigl[\Gamma_{p,m}\bigl(Z^{(q,H)}\bigr)\bigr]
  \;\le\; C_{p,q}\int_0^T\!\!\int_0^T |t-s|^{pH-m-2}\,ds\,dt
  \;=\; \frac{2\,C_{p,q}\,T^{\,pH-m}}{(pH-m-1)(pH-m)},
\]
which is finite because $pH-m-2>-1$, i.e.\ $m<pH-1$. Combining this
with~\eqref{eq:Kp_bound} yields the claim. The H\"older exponent
obtained, $m/p$, ranges over $(0,H-1/p)$, as in Kolmogorov's
criterion.
\end{proof}

\begin{corollary}[{Finiteness of $\E[C_T^p]$}]
\label{cor:CT_moments}
$\E\bigl[C_T^p\bigr]<\infty$ for
every $p\geq 1$, where $C_T=\sup_{s\le T}|Z_s^{(q,H)}|$.
\end{corollary}

\begin{proof}
Fix an integer $p_0 \ge p$ with $p_0 H>1$ and some
$m\in(0,p_0H-1)$. Since $Z_0^{(q,H)}=0$ a.s.\
(Definition~\ref{def:hermite}), Lemma~\ref{lem:grr} gives
\[
  C_T \;=\; \sup_{t\le T}\bigl|Z_t^{(q,H)}-Z_0^{(q,H)}\bigr|
  \;\le\; K_{p_0,m}\, T^{m/p_0} \qquad \text{a.s.},
\]
so that
\[
  \E\bigl[C_T^{p_0}\bigr] \;\le\; T^{m}\,\E\bigl[K_{p_0,m}^{p_0}\bigr]
  \;\le\; \frac{2\,c_{p_0,m}\, C_{p_0,q}}{(p_0H-m-1)(p_0H-m)}\, T^{p_0 H} \;<\;\infty .
\]
By Lyapunov's inequality,
$\bigl(\E[C_T^p]\bigr)^{1/p}\le\bigl(\E[C_T^{p_0}]\bigr)^{1/p_0}<\infty$.
Finally, by the $H$-self-similarity of $Z^{(q,H)}$
(Proposition~\ref{prop:hermite_properties}(ii)), $C_T$ has the same law
as $T^H C_1$, so in fact
$\E[C_T^p]=T^{pH}\,\E[C_1^p]$ for every $p\ge1$, in
agreement with the power of $T$ obtained above.
\end{proof}

The next lemma gives the dependence of $\E[C_1^p]$ on $p$. It is the
ingredient behind the stretched-exponential form of the positivity
bounds of Section~\ref{sec:positivity}. It is an instance of the
hypercontractivity of Wiener chaos with values in a Banach space
(see~\cite{ArconesGine1993} and~\cite[Section~3.2]{LedouxTalagrand1991});
we include the short argument for completeness.

\begin{lemma}[Moment growth of the supremum]
\label{lem:moment_growth}
Let $C_1$ be the supremum of $|Z^{(q,H)}|$ on $[0,1]$, and let
$\kappa:=(\E[C_1^2])^{1/2}$. By Corollary~\ref{cor:CT_moments},
$\kappa<\infty$; it depends only on $q$ and $H$. Then,
for every real $p\geq2$,
\begin{equation}
\label{eq:C1_hyper}
\bigl(\E\bigl[C_1^{\,p}\bigr]\bigr)^{1/p}\;\le\;(p-1)^{q/2}\,\kappa .
\end{equation}
Consequently, for every $x>0$,
\begin{equation}
\label{eq:C1_tail}
\PP\bigl(C_1\geq x\bigr)\;\le\;
e^{2}\exp\!\Bigl(-\Bigl(\frac{x}{e\,\kappa}\Bigr)^{2/q}\Bigr).
\end{equation}
\end{lemma}

\begin{proof}
Let $E:=C([0,1];\R)$ with the uniform norm $\|\cdot\|_\infty$, a
separable Banach space, so that $C_1=\|Z^{(q,H)}\|_\infty$. By
Definition~\ref{def:hermite} the continuous process $Z^{(q,H)}$ is a
measurable functional of the isonormal process $W$ underlying the
multiple integrals; we write $Z^{(q,H)}=F(W)$ with
$F$ measurable with values in $E$ and $\E\|F(W)\|_\infty^2=\kappa^2<\infty$.
Let $W'$ be an independent copy of $W$ and, for $\tau\ge0$, let
\[
(P_\tau F)(W):=\E'\Bigl[F\bigl(e^{-\tau}W+\sqrt{1-e^{-2\tau}}\,W'\bigr)\Bigr],
\]
where $\E'$ is the expectation with respect to $W'$ only. For each
$s\in[0,1]$, this is Mehler's formula for the Ornstein--Uhlenbeck
semigroup applied to the real random variable
$Z_s^{(q,H)}\in\mathcal H_q$~\cite[Section~1.4.1]{Nualart}, and
$P_\tau$ acts on the $q$-th chaos as multiplication by $e^{-q\tau}$.
Hence $(P_\tau F)(W)(s)=e^{-q\tau}Z_s^{(q,H)}$ almost surely, for each
$s$. Both sides are continuous in $s$ (the left-hand side by dominated
convergence, since $\E'\|F(e^{-\tau}W+\sqrt{1-e^{-2\tau}}W')\|_\infty<\infty$
almost surely), so they agree as elements of $E$, almost surely:
\[
(P_\tau F)(W)=e^{-q\tau}Z^{(q,H)} .
\]
On the other hand, by the triangle inequality for the integral,
\[
\bigl\|(P_\tau F)(W)\bigr\|_\infty
\le\E'\Bigl[\bigl\|F\bigl(e^{-\tau}W+\sqrt{1-e^{-2\tau}}\,W'\bigr)\bigr\|_\infty\Bigr]
=\bigl(P_\tau\|F\|_\infty\bigr)(W),
\]
where on the right $P_\tau$ acts on the non-negative real random variable
$\|F(W)\|_\infty=C_1\in L^2(\Omega)$. Choose $\tau$ with
$e^{2\tau}=p-1$. By Nelson's hypercontractivity
inequality~\cite[Theorem~1.4.1]{Nualart},
$\|P_\tau G\|_{L^p(\Omega)}\le\|G\|_{L^2(\Omega)}$ for every
$G\in L^2(\Omega)$. Combining the last three displays,
\[
e^{-q\tau}\bigl(\E[C_1^p]\bigr)^{1/p}
=\bigl\|\,\|(P_\tau F)(W)\|_\infty\bigr\|_{L^p(\Omega)}
\le\bigl\|P_\tau C_1\bigr\|_{L^p(\Omega)}
\le\bigl(\E[C_1^2]\bigr)^{1/2}=\kappa ,
\]
and $e^{q\tau}=(p-1)^{q/2}$, which is~\eqref{eq:C1_hyper}.

For~\eqref{eq:C1_tail}, fix $x>0$ and set $p_*:=(x/(e\kappa))^{2/q}$.
If $p_*\geq2$, Markov's inequality with $p=p_*$
and~\eqref{eq:C1_hyper} give
\[
\PP(C_1\ge x)\le\frac{\E[C_1^{p_*}]}{x^{p_*}}
\le\Bigl(\frac{p_*^{q/2}\,\kappa}{x}\Bigr)^{p_*}
=\Bigl(\frac{x}{e\kappa}\cdot\frac{\kappa}{x}\Bigr)^{p_*}
=e^{-p_*}\le e^{2}e^{-p_*}.
\]
If $p_*<2$, the right-hand side of~\eqref{eq:C1_tail} is
$e^{2-p_*}>1$, and there is nothing to prove.
\end{proof}

\begin{remark}[Sharpness of the exponent $2/q$ for $q\le2$]
\label{rem:tail_sharp}
Since $C_1\geq|Z_1^{(q,H)}|$, the tail of $C_1$ is at least as heavy as
that of $Z_1^{(q,H)}$. For $q=1$, $Z_1^{(1,H)}\sim\mathcal N(0,1)$ has
tails of order $\exp(-x^2/2)$. For $q=2$, $Z_1^{(2,H)}=\sum_k\mu_k(\xi_k^2-1)$
with $(\xi_k)$ i.i.d.\ $\mathcal N(0,1)$ and $\mu_1>0$ the largest
eigenvalue of the associated Hilbert--Schmidt operator, so, by
independence,
$\PP(Z_1^{(2,H)}\ge x)\ge\PP(\mu_1(\xi_1^2-1)\ge x)\,
\PP\bigl(\sum_{k\ge2}\mu_k(\xi_k^2-1)\ge0\bigr)$, which is of
order $\exp(-x/(2\mu_1))$ up to a polynomial factor (the second factor
is a positive constant, the sum being centred). In both cases
the exponent $2/q$ in~\eqref{eq:C1_tail} cannot be improved. We do not
know whether the same holds for $q\ge3$.
\end{remark}
\section{Existence and uniqueness}
\label{sec:existence}

We use the deterministic theorem of Nualart and
R\u{a}\c{s}canu~\cite[Theorem~5.1]{MR1893308} for the equation
\begin{equation}
\label{eq:NR_equation}
x_t=x_0+\int_0^tb^\sharp(s,x_s)\,ds+\int_0^t\sigma^\sharp(s,x_s)\,dw_s,
\qquad t\in[0,T],
\end{equation}
where $w\in W_2^{1-\alpha}(0,T)$ and the second integral is the
generalised Stieltjes integral of Section~\ref{subsec:young}. In the form
in which we use it, the theorem requires the following.
\begin{enumerate}[label=(NR\arabic*)]
\item $\sigma^\sharp(t,x)$ is differentiable in $x$, and there are
constants $M_0$, $\beta,\delta\in(0,1]$ and, for every $N\ge0$, $M_N$
such that, for all $t,s\in[0,T]$ and $x,y\in\R$,
\begin{align*}
&|\sigma^\sharp(t,x)-\sigma^\sharp(t,y)|\le M_0|x-y|,\\
&|\partial_x\sigma^\sharp(t,x)-\partial_y\sigma^\sharp(t,y)|\le M_N|x-y|^\delta
\quad\text{if }|x|,|y|\le N,\\
&|\sigma^\sharp(t,x)-\sigma^\sharp(s,x)|+|\partial_x\sigma^\sharp(t,x)-\partial_x\sigma^\sharp(s,x)|\le M_0|t-s|^\beta .
\end{align*}
\item $b^\sharp$ is measurable, locally Lipschitz in $x$ uniformly in $t$,
and $|b^\sharp(t,x)|\le L_0|x|+b^\sharp_0(t)$ for a constant $L_0$ and a
function $b^\sharp_0\in L^\rho(0,T)$ with $\rho\ge1/\alpha$.
\end{enumerate}
Under (NR1) and (NR2), with $\alpha_0:=\min(\frac12,\beta,\frac\delta{1+\delta})$
and $0<\alpha<\alpha_0$, equation~\eqref{eq:NR_equation} has a unique
solution in $W_1^\alpha(0,T)$ for every $w\in W_2^{1-\alpha}(0,T)$, and
this solution belongs to $C^{1-\alpha}(0,T)$~\cite[Theorem~5.1]{MR1893308}.

\begin{theorem}
\label{thm:existence}
Let $T>0$. There is an event $\Omega'$ of probability one such that,
for every $\omega\in\Omega'$ and every $\alpha\in(1-H,\frac12)$,
equation~\eqref{eq:main_sde} has a unique solution $X(\omega)$ in
$W_1^\alpha(0,T)$. This solution does not depend on $\alpha$ and belongs to
$C^{\gamma}(0,T)$ for every $\gamma<H$. Setting
$X:=x_0$ on $\Omega\setminus\Omega'$, $X_t$ is a random variable for every
$t\in[0,T]$ (Remark~\ref{rem:measurability}).
\end{theorem}

\begin{proof}
\emph{Step~1: the driving path.}
By Proposition~\ref{prop:hermite_properties}(v), there is an event
$\Omega'$ of probability one on which the path of $Z^{(q,H)}$ is
$\zeta$-H\"older continuous on $[0,T]$ for every $\zeta<H$. Fix
$\alpha\in(1-H,\frac12)$ and set $\beta':=\alpha-(1-H)\in(0,H-\frac12)$,
so that $1-\alpha=H-\beta'$. Choosing $\zeta\in(H-\beta',H)$, the
computation of Step~3 of the proof of Lemma~\ref{lem:hermite_regularity}
shows that, on $\Omega'$, the path of $Z^{(q,H)}$ belongs to
$W_2^{1-\alpha}(0,T)$, for every such $\alpha$ simultaneously.

\medskip
\noindent\emph{Step~2: the coefficients.}
Here $\sigma^\sharp(t,x)=\tilde\sigma(x)$ and $b^\sharp(t,x)=a(b(t)-x)$.
By Lemma~\ref{lem:sigma}, $\tilde\sigma$ is globally Lipschitz and
$\tilde\sigma'$ is globally Lipschitz, since $\tilde\sigma''$ is bounded;
$\tilde\sigma$ does not depend on time. Hence (NR1) holds with $\delta=1$
and $\beta=1$ (the third inequality is trivial), with
$M_0=M_N=\max(\|\tilde\sigma'\|_\infty,\|\tilde\sigma''\|_\infty)$. The
drift is globally Lipschitz in $x$, uniformly in $t$, and
\[
\bigl|a\bigl(b(t)-x\bigr)\bigr|\le a|x|+a\|b\|_\infty,
\qquad t\in[0,T],\ x\in\R,
\]
so (NR2) holds with $L_0=a$ and the constant function
$b_0^\sharp\equiv a\|b\|_\infty$, which belongs to every $L^\rho(0,T)$.
With these exponents, $\alpha_0=\min(\frac12,1,\frac12)=\frac12$, so every
$\alpha\in(1-H,\frac12)$ is admissible.

\medskip
\noindent\emph{Step~3: existence and uniqueness for a fixed $\alpha$.}
On $\Omega'$, we apply the theorem recalled above with $w$ the path of
$Z^{(q,H)}$. For every $\alpha\in(1-H,\frac12)$, this gives a unique
solution $X^{(\alpha)}\in W_1^\alpha(0,T)\cap C^{1-\alpha}(0,T)$.

\medskip
\noindent\emph{Step~4: independence of $\alpha$.}
Let $1-H<\alpha<\alpha'<\frac12$. First, $W_1^{\alpha'}(0,T)\subset
W_1^{\alpha}(0,T)$, because $(t-s)^{-\alpha-1}\le
T^{\alpha'-\alpha}(t-s)^{-\alpha'-1}$ for $0\le s<t\le T$. Second, for
$f\in W_1^{\alpha'}(0,T)$ and $g\in W_2^{1-\alpha}(0,T)\subset
W_2^{1-\alpha'}(0,T)$, the generalised Stieltjes integrals $\int f\,dg$
defined with the parameters $\alpha$ and $\alpha'$ coincide: both are
equal to the Riemann--Stieltjes integral when $f$ and $g$ are H\"older
continuous with exponents adding up to more than $1$, which is the case
here (Section~\ref{subsec:young} and~\cite{zahle1998}). Hence $X^{(\alpha')}$ is a solution
in $W_1^\alpha(0,T)$ of the equation defined with the parameter $\alpha$,
and uniqueness in $W_1^\alpha(0,T)$ gives $X^{(\alpha')}=X^{(\alpha)}$. We
denote this common solution by $X$. It belongs to $C^{1-\alpha}(0,T)$ for
every $\alpha\in(1-H,\frac12)$, that is, to $C^\gamma(0,T)$ for every
$\gamma<H$.

\medskip
\noindent\emph{Step~5: measurability.} On $\Omega\setminus\Omega'$, which is
a null set, we set $X:=x_0$. The measurability of $X_t$ is proved in
Remark~\ref{rem:measurability} below, after the change of variables of
Section~\ref{subsec:lamperti}; this change of variables only uses the
pathwise equation on $\Omega'$.
\end{proof}

\begin{remark}
\label{rem:existence_alt}
For $\sigma_0>0$, existence and uniqueness can also be obtained directly
from the change of variables of Section~\ref{subsec:lamperti}, in the
spirit of Doss and Sussmann~\cite{Doss1977,Sussmann1978}: the equation
for $Y=\Lambda(X)$ is an ordinary differential equation with a globally
Lipschitz drift, driven additively by $Z^{(q,H)}$
(Remark~\ref{rem:measurability}). We use the theorem of Nualart and
R\u{a}\c{s}canu because it gives the solution in the spaces of
Section~\ref{subsec:young} directly.
\end{remark}

\subsection{An equivalent equation with additive noise}
\label{subsec:lamperti}

Since $\tilde\sigma\geq\sigma_0>0$, the noise in~\eqref{eq:main_sde} can
be made additive by a change of variables of Lamperti type~\cite{Lamperti1964}, as in the Doss--Sussmann approach~\cite{Doss1977,Sussmann1978}. This is used
both for positivity (Section~\ref{sec:positivity}) and for the Malliavin
analysis (Section~\ref{sec:malliavin}). Define
\begin{equation}
\label{eq:Lambda}
\Lambda(x):=\int_0^x\frac{du}{\tilde\sigma(u)},\qquad x\in\R .
\end{equation}
By Lemma~\ref{lem:sigma}, $\Lambda\in C^\infty(\R)$,
$\Lambda(0)=0$, $\Lambda'=1/\tilde\sigma\in(0,1/\sigma_0]$, so that $\Lambda$ is
strictly increasing, and
\[
\Lambda''=-\frac{\tilde\sigma'}{\tilde\sigma^{2}},
\qquad
\|\Lambda''\|_\infty\le\frac{\|\tilde\sigma'\|_\infty}{\sigma_0^{2}}<\infty .
\]

\begin{lemma}[Additive form of the equation]
\label{lem:lamperti}
Let $X$ be the solution
of Theorem~\ref{thm:existence} and set $Y_t:=\Lambda(X_t)$. Then, almost
surely, for every $t\in[0,T]$,
\begin{equation}
\label{eq:lamperti_Y}
Y_t=\Lambda(x_0)+\int_0^t\frac{a\bigl(b(s)-X_s\bigr)}{\tilde\sigma(X_s)}\,ds
+Z_t^{(q,H)} .
\end{equation}
Moreover, $\{X_t>0\}=\{Y_t>0\}$ for every $t\in[0,T]$.
\end{lemma}

\begin{proof}
Fix $\alpha\in(1-H,\tfrac12)$ and set $\gamma:=1-\alpha\in(\tfrac12,H)$.
We work on the event of probability one on which $Z^{(q,H)}$ belongs to
$W_2^{1-\alpha}(0,T;\R)\subset C^{\gamma}(0,T;\R)$
(Step~1 of the proof of Theorem~\ref{thm:existence}) and $X$ solves~\eqref{eq:main_sde} with
paths in $W_1^\alpha\cap C^{\gamma}$; we fix such a trajectory and
$t\in(0,T]$. We write $[f]_\gamma$ for the $\gamma$-H\"older seminorm
of a function $f$ on $[0,T]$. Since $\tilde\sigma$ is Lipschitz,
$\tilde\sigma(X)\in C^\gamma$ with
$[\tilde\sigma(X)]_\gamma\le\|\tilde\sigma'\|_\infty[X]_\gamma$.
Because $2\gamma>1$, the integral
$\int\tilde\sigma(X_s)\,dZ_s^{(q,H)}$ of~\eqref{eq:main_sde} is a Young
integral, and~\eqref{eq:young_loeve} gives, with
$c_\gamma:=c_{\gamma,\gamma}$, for all $0\le u<v\le T$,
\begin{equation}
\label{eq:young_loeve_X}
\Bigl|\int_u^v\tilde\sigma(X_s)\,dZ_s^{(q,H)}
-\tilde\sigma(X_u)\bigl(Z_v^{(q,H)}-Z_u^{(q,H)}\bigr)\Bigr|
\le c_\gamma\,[\tilde\sigma(X)]_\gamma\,[Z^{(q,H)}]_\gamma\,(v-u)^{2\gamma}.
\end{equation}

Let $\pi=\{0=t_0<t_1<\dots<t_n=t\}$ be a partition of $[0,t]$ with mesh
$|\pi|$, and write $\Delta_k X:=X_{t_k}-X_{t_{k-1}}$,
$\Delta_kZ:=Z^{(q,H)}_{t_k}-Z^{(q,H)}_{t_{k-1}}$. By Taylor's formula,
for each $k$ there is $x_k^*$ between $X_{t_{k-1}}$ and $X_{t_k}$ such
that
\begin{align}
\Lambda(X_t)-\Lambda(x_0)
&=\sum_{k=1}^n\bigl(\Lambda(X_{t_k})-\Lambda(X_{t_{k-1}})\bigr)\notag\\
&=\sum_{k=1}^n\Lambda'(X_{t_{k-1}})\,\Delta_kX
+\frac12\sum_{k=1}^n\Lambda''(x_k^*)\,(\Delta_kX)^2\notag\\
&=:S_1(\pi)+S_2(\pi).\label{eq:lamperti_taylor}
\end{align}

\emph{The second-order term.} Since $|\Delta_kX|\le[X]_\gamma(t_k-t_{k-1})^\gamma$,
\begin{align*}
|S_2(\pi)|
&\le\frac{\|\Lambda''\|_\infty}{2}\,[X]_\gamma^2\sum_{k=1}^n(t_k-t_{k-1})^{2\gamma}\\
&\le\frac{\|\Lambda''\|_\infty}{2}\,[X]_\gamma^2\,|\pi|^{2\gamma-1}\sum_{k=1}^n(t_k-t_{k-1})\\
&=\frac{\|\Lambda''\|_\infty}{2}\,[X]_\gamma^2\,t\,|\pi|^{2\gamma-1},
\end{align*}
which tends to $0$ as $|\pi|\to0$, because $2\gamma>1$.

\emph{The first-order term.} By~\eqref{eq:main_sde},
\[
\Delta_kX=\int_{t_{k-1}}^{t_k}a\bigl(b(s)-X_s\bigr)\,ds
+\int_{t_{k-1}}^{t_k}\tilde\sigma(X_s)\,dZ_s^{(q,H)},
\]
so that $S_1(\pi)=S_{1,1}(\pi)+S_{1,2}(\pi)+S_{1,3}(\pi)$ with
\begin{align*}
S_{1,1}(\pi)&:=\sum_{k=1}^n\Lambda'(X_{t_{k-1}})\int_{t_{k-1}}^{t_k}a\bigl(b(s)-X_s\bigr)\,ds,\\
S_{1,2}(\pi)&:=\sum_{k=1}^n\Lambda'(X_{t_{k-1}})\,\tilde\sigma(X_{t_{k-1}})\,\Delta_kZ,\\
S_{1,3}(\pi)&:=\sum_{k=1}^n\Lambda'(X_{t_{k-1}})\Bigl(\int_{t_{k-1}}^{t_k}\tilde\sigma(X_s)\,dZ_s^{(q,H)}
-\tilde\sigma(X_{t_{k-1}})\,\Delta_kZ\Bigr).
\end{align*}
Denote by $\omega(\delta):=\sup\{|\Lambda'(X_u)-\Lambda'(X_v)|:\ u,v\in[0,t],\
|u-v|\le\delta\}$ the modulus of continuity of the continuous function
$s\mapsto\Lambda'(X_s)$, so that $\omega(\delta)\to0$ as $\delta\to0$. Then
\begin{align*}
\Bigl|S_{1,1}(\pi)-\int_0^t\Lambda'(X_s)\,a\bigl(b(s)-X_s\bigr)\,ds\Bigr|
&=\Bigl|\sum_{k=1}^n\int_{t_{k-1}}^{t_k}\bigl(\Lambda'(X_{t_{k-1}})-\Lambda'(X_s)\bigr)
\,a\bigl(b(s)-X_s\bigr)\,ds\Bigr|\\
&\le\omega(|\pi|)\int_0^t\bigl|a\bigl(b(s)-X_s\bigr)\bigr|\,ds
\;\longrightarrow\;0,
\end{align*}
since the drift is bounded on $[0,t]$ along the fixed trajectory. Hence
\[
S_{1,1}(\pi)\;\longrightarrow\;\int_0^t\Lambda'(X_s)\,a\bigl(b(s)-X_s\bigr)\,ds
=\int_0^t\frac{a\bigl(b(s)-X_s\bigr)}{\tilde\sigma(X_s)}\,ds .
\]
Since $\Lambda'\tilde\sigma\equiv1$, the second sum telescopes exactly:
\[
S_{1,2}(\pi)=\sum_{k=1}^n\Delta_kZ=Z_t^{(q,H)}-Z_0^{(q,H)}=Z_t^{(q,H)} .
\]
Finally, by~\eqref{eq:young_loeve_X} and $|\Lambda'|\le1/\sigma_0$,
\begin{align*}
|S_{1,3}(\pi)|
&\le\frac{c_\gamma}{\sigma_0}\,[\tilde\sigma(X)]_\gamma\,[Z^{(q,H)}]_\gamma
\sum_{k=1}^n(t_k-t_{k-1})^{2\gamma}\\
&\le\frac{c_\gamma}{\sigma_0}\,[\tilde\sigma(X)]_\gamma\,[Z^{(q,H)}]_\gamma\,
t\,|\pi|^{2\gamma-1}
\;\longrightarrow\;0 .
\end{align*}
Letting $|\pi|\to0$ in~\eqref{eq:lamperti_taylor} yields~\eqref{eq:lamperti_Y}
for the fixed $t$. Since the argument holds for every $t\in(0,T]$ on
the same event of probability one, fixed at the beginning of the
proof, the identity holds for every $t\in[0,T]$ simultaneously. The last claim follows
from the fact that $\Lambda$ is strictly increasing with $\Lambda(0)=0$.
\end{proof}

The drift of $Y=\Lambda(X)$ can be controlled in terms of $Y$ itself. For
$t\in\R$ and $x\in\R$, set
\begin{equation}
\label{eq:g_h}
\psi(t,x):=\frac{a\bigl(b(t)-x\bigr)}{\tilde\sigma(x)},
\qquad
g(t,x):=\frac{\tilde\sigma'(x)\bigl(b(t)-x\bigr)}{\tilde\sigma(x)},
\qquad
k(t,y):=\psi\bigl(t,\Lambda^{-1}(y)\bigr),
\end{equation}
where $\Lambda^{-1}$ is the inverse of $\Lambda$, which is a bijection
of $\R$ by the last part of the next lemma.

\begin{lemma}
\label{lem:drift_bounds}
\begin{enumerate}[label=(\roman*)]
\item For all $t,x$, $|g(t,x)|\le G$, where
\[
G:=\max\Bigl(\frac12+\frac{\|b\|_\infty}{2\eps},\;
\frac{\|\tilde\sigma'\|_\infty\,(\|b\|_\infty+\eps)}{\sigma_0}\Bigr),
\]
and $g\equiv0$ if $\sigma_1=0$.
\item The function $k$ is continuous on $\R^2$, and for all $t,y,y'$,
\[
|k(t,y)-k(t,y')|\le a(1+G)\,|y-y'|,
\qquad
|k(t,y)|\le\frac{a\|b\|_\infty}{\sigma_0}+a(1+G)\,|y| .
\]
\item With $c_0:=\sigma_0+\sigma_1\sqrt\eps$, for every $x\in\R$,
\[
|x|\le2\sigma_1^2\,\Lambda(x)^2+2c_0\,|\Lambda(x)| .
\]
In particular $|\Lambda(x)|\to\infty$ as $|x|\to\infty$, and $\Lambda$ is
a bijection of $\R$.
\end{enumerate}
\end{lemma}

\begin{proof}
(i) If $\sigma_1=0$, then $\tilde\sigma'\equiv0$ and $g\equiv0$. Let
$\sigma_1>0$. For $x\geq0$, Lemma~\ref{lem:sigma} gives
$\tilde\sigma'(x)=\sigma_1/(2\sqrt{x+\eps})$ and
$\tilde\sigma(x)\geq\sigma_1\sqrt{x+\eps}\geq\sigma_1\sqrt\eps$, so that
\[
|g(t,x)|
\le\frac{\sigma_1\,x}{2\sqrt{x+\eps}\,\sigma_1\sqrt{x+\eps}}
+\frac{\sigma_1\,\|b\|_\infty}{2\sqrt{\eps}\,\sigma_1\sqrt{\eps}}
=\frac{x}{2(x+\eps)}+\frac{\|b\|_\infty}{2\eps}
\le\frac12+\frac{\|b\|_\infty}{2\eps}.
\]
For $x\le-\eps$, $\phi_\eps$ is constant on $(-\infty,-\eps]$ (Lemma~\ref{lem:phi}(a)) and smooth, so $\tilde\sigma'(x)=0$ and $g(t,x)=0$. For $x\in(-\eps,0)$,
$|b(t)-x|\le\|b\|_\infty+\eps$ and $\tilde\sigma(x)\geq\sigma_0$, so
$|g(t,x)|\le\|\tilde\sigma'\|_\infty(\|b\|_\infty+\eps)/\sigma_0$.

(ii) Continuity is clear. Since $\Lambda'=1/\tilde\sigma$, the inverse
function satisfies $(\Lambda^{-1})'(y)=\tilde\sigma(\Lambda^{-1}(y))$. A
direct computation gives
\[
\partial_x\psi(t,x)=\frac{-a\,\tilde\sigma(x)-a\bigl(b(t)-x\bigr)\tilde\sigma'(x)}{\tilde\sigma(x)^2}
=-\frac{a\bigl(1+g(t,x)\bigr)}{\tilde\sigma(x)},
\]
hence, with $x=\Lambda^{-1}(y)$,
\[
\partial_yk(t,y)=\partial_x\psi(t,x)\,\tilde\sigma(x)=-a\bigl(1+g(t,x)\bigr),
\qquad
|\partial_yk(t,y)|\le a(1+G)
\]
by (i). This is the Lipschitz bound. The growth bound follows from it
and from $|k(t,0)|=|\psi(t,0)|=a|b(t)|/\tilde\sigma(0)\le a\|b\|_\infty/\sigma_0$.

(iii) Let $x\geq0$. Since $\tilde\sigma$ is nondecreasing on
$[0,\infty)$ (Lemma~\ref{lem:sigma}),
$\Lambda(x)=\int_0^xdu/\tilde\sigma(u)\geq x/\tilde\sigma(x)$, and
$\tilde\sigma(x)=\sigma_0+\sigma_1\sqrt{x+\eps}\le c_0+\sigma_1\sqrt x$.
Hence, writing $r=\sqrt x$ and $\Lambda=\Lambda(x)\geq0$,
\[
r^2-\sigma_1\Lambda\,r-c_0\Lambda\le0,
\qquad\text{so}\qquad
r\le\frac{\sigma_1\Lambda+\sqrt{\sigma_1^2\Lambda^2+4c_0\Lambda}}{2}
\le\sigma_1\Lambda+\sqrt{c_0\Lambda},
\]
and therefore $x=r^2\le2\sigma_1^2\Lambda^2+2c_0\Lambda$. Let now $x<0$.
By Lemma~\ref{lem:sigma}, $\tilde\sigma\le c_0$ on $(-\infty,0]$, so
$|\Lambda(x)|=\int_x^0du/\tilde\sigma(u)\geq|x|/c_0$, i.e.\
$|x|\le c_0|\Lambda(x)|$. In both cases the claimed inequality holds. It
implies $|\Lambda(x)|\to\infty$ as $|x|\to\infty$; since $\Lambda$ is
continuous and strictly increasing, it is a bijection of $\R$.
\end{proof}

\begin{remark}[Measurability of the solution]
\label{rem:measurability}
Fix $\omega$ in the event of probability one of
Lemma~\ref{lem:lamperti}. By~\eqref{eq:lamperti_Y}, the function
$U:=Y-Z^{(q,H)}$ solves the ordinary differential equation
\[
U_t=\Lambda(x_0)+\int_0^tk\bigl(s,U_s+Z_s^{(q,H)}\bigr)\,ds,
\qquad t\in[0,T],
\]
whose right-hand side is continuous in $s$ and, by
Lemma~\ref{lem:drift_bounds}(ii), globally Lipschitz in $U$. Its
solution is unique and is the uniform limit of the Picard iterates
$U^{(0)}\equiv\Lambda(x_0)$,
$U^{(n+1)}_t=\Lambda(x_0)+\int_0^tk(s,U^{(n)}_s+Z_s^{(q,H)})\,ds$. Each
$U^{(n)}_t$ is a continuous function of the path $(Z_s^{(q,H)})_{s\le t}$
for the uniform norm. Hence $X_t=\Lambda^{-1}(U_t+Z_t^{(q,H)})$ is a
measurable function of $(Z_s^{(q,H)})_{s\le t}$, and $X$ is adapted to
the filtration generated by $Z^{(q,H)}$.
\end{remark}

\begin{lemma}
\label{lem:X_moments}
For every $p\geq1$, $\E\bigl[\sup_{t\le T}|X_t|^p\bigr]<\infty$.
\end{lemma}

\begin{proof}
By~\eqref{eq:lamperti_Y} and Lemma~\ref{lem:drift_bounds}(ii), for
every $t\in[0,T]$,
\[
|Y_t|\le|\Lambda(x_0)|+\int_0^t\Bigl(\frac{a\|b\|_\infty}{\sigma_0}+a(1+G)\,|Y_s|\Bigr)ds+C_T,
\]
so that, by Gronwall's lemma,
\[
\sup_{t\le T}|Y_t|\le\Bigl(|\Lambda(x_0)|+\frac{a\|b\|_\infty T}{\sigma_0}+C_T\Bigr)e^{a(1+G)T}.
\]
By Lemma~\ref{lem:drift_bounds}(iii),
$\sup_{t\le T}|X_t|\le2\sigma_1^2\sup_{t\le T}Y_t^2+2c_0\sup_{t\le T}|Y_t|$,
and the claim follows from Corollary~\ref{cor:CT_moments}.
\end{proof}

\section{Positivity}
\label{sec:positivity}

Throughout this section we assume~\eqref{eq:target_positivity}, and we set
\[
m_0:=\min(x_0,b_*)>0 .
\]
We do not claim that $X$ stays positive almost surely: this is false at
least for $q\le2$ and $\sigma_1=0$ (Proposition~\ref{prop:no_as}). We give
lower bounds on the
probability that $X$ stays positive on $[0,T]$, first for $\sigma_1=0$,
where the solution is explicit, and then for every $\sigma_1\geq0$.

\subsection{The case $\sigma_1=0$}

\begin{theorem}
\label{thm:positivity}
Assume~\eqref{eq:target_positivity} and $\sigma_1=0$. Then:
\begin{enumerate}[label=(\roman*)]
\item for every $p\geq1$,
\begin{equation}
\label{eq:positivity_quant}
\PP\bigl(X_t>0\ \text{for all }t\in[0,T]\bigr)
\;\geq\;
1-\left(\frac{\sigma_0\,(2-e^{-aT})\,T^{H}}{m_0}\right)^{p}\E\bigl[C_1^{\,p}\bigr];
\end{equation}
\item with the constant $\kappa$ of Lemma~\ref{lem:moment_growth},
\begin{equation}
\label{eq:positivity_exp}
\PP\bigl(X_t>0\ \text{for all }t\in[0,T]\bigr)
\;\geq\;
1-e^{2}\exp\!\left(-\left(\frac{m_0}{e\,\kappa\,\sigma_0\,(2-e^{-aT})\,T^{H}}\right)^{2/q}\right).
\end{equation}
\end{enumerate}
\end{theorem}

Since $\E[C_1^p]<\infty$ for every $p$ (Corollary~\ref{cor:CT_moments}),
both bounds tend to $1$ as $m_0/\sigma_0\to\infty$. By (ii), the
probability of leaving $(0,\infty)$ before $T$ is at most of order
$\exp(-c\,(m_0/\sigma_0)^{2/q})$; this upper bound decays more slowly
for larger $q$, in line with the heavier tails of the $q$-th Wiener
chaos.

\begin{proof}
\emph{Step~1: explicit solution.}
For $\sigma_1=0$ the noise is additive:
$X_t=x_0+\int_0^ta(b(s)-X_s)\,ds+\sigma_0Z_t^{(q,H)}$. The process
$U:=X-\sigma_0Z^{(q,H)}$ is therefore continuously differentiable, with
$U_0=x_0$ and
\[
U_t'=a\bigl(b(t)-U_t-\sigma_0Z_t^{(q,H)}\bigr),
\]
a linear ordinary differential equation whose solution is
\[
U_t=e^{-at}x_0+\int_0^te^{-a(t-s)}a\,b(s)\,ds-a\sigma_0\int_0^te^{-a(t-s)}Z_s^{(q,H)}\,ds .
\]
Hence
\begin{equation}
\label{eq:X_linear_explicit}
X_t=e^{-at}x_0+\underbrace{\int_0^te^{-a(t-s)}a\,b(s)\,ds}_{=:M_t}
+\underbrace{\sigma_0\Bigl(Z_t^{(q,H)}-a\int_0^te^{-a(t-s)}Z_s^{(q,H)}\,ds\Bigr)}_{=:I_t},
\qquad t\in[0,T].
\end{equation}

\medskip
\noindent\emph{Step~2: the deterministic part.}
Since $b(s)\geq b_*$ for every $s$ and
$\int_0^ta\,e^{-a(t-s)}\,ds=1-e^{-at}$,
\[
e^{-at}x_0+M_t
\;\geq\;e^{-at}x_0+b_*\bigl(1-e^{-at}\bigr)
\;\geq\;\min(x_0,b_*)=m_0,
\qquad t\in[0,T],
\]
because the middle expression is a convex combination of $x_0$ and
$b_*$.

\medskip
\noindent\emph{Step~3: the noise part.}
For every $t\in[0,T]$, using $|Z_s^{(q,H)}|\le C_T$ for $s\le t$,
\[
|I_t|
\le\sigma_0\Bigl(\bigl|Z_t^{(q,H)}\bigr|+a\int_0^te^{-a(t-s)}\bigl|Z_s^{(q,H)}\bigr|\,ds\Bigr)
\le\sigma_0\bigl(1+1-e^{-at}\bigr)C_T
\le\sigma_0\bigl(2-e^{-aT}\bigr)C_T .
\]

\medskip
\noindent\emph{Step~4: probabilities.}
By Steps 1 to 3, $X_t\geq m_0-\sigma_0(2-e^{-aT})C_T$ for every
$t\in[0,T]$, so that
\begin{equation}
\label{eq:event_inclusion}
\Bigl\{C_T<\frac{m_0}{\sigma_0(2-e^{-aT})}\Bigr\}
\subseteq
\bigl\{X_t>0\ \text{for all }t\in[0,T]\bigr\}.
\end{equation}
By Markov's inequality and Corollary~\ref{cor:CT_moments}, for every
$p\geq1$,
\[
\PP\Bigl(C_T\geq\frac{m_0}{\sigma_0(2-e^{-aT})}\Bigr)
\le\left(\frac{\sigma_0(2-e^{-aT})}{m_0}\right)^{p}\E\bigl[C_T^{\,p}\bigr]
=\left(\frac{\sigma_0(2-e^{-aT})\,T^H}{m_0}\right)^{p}\E\bigl[C_1^{\,p}\bigr],
\]
which gives~\eqref{eq:positivity_quant}. Since $C_T$ has the law of
$T^HC_1$, \eqref{eq:event_inclusion} and~\eqref{eq:C1_tail} applied with
$x=m_0/(\sigma_0(2-e^{-aT})T^H)$ give
\begin{align*}
\PP\bigl(\exists\,t\le T:\ X_t\le0\bigr)
&\le\PP\Bigl(C_1\geq\frac{m_0}{\sigma_0(2-e^{-aT})T^H}\Bigr)\\
&\le e^{2}\exp\!\left(-\left(\frac{m_0}{e\,\kappa\,\sigma_0\,(2-e^{-aT})\,T^{H}}\right)^{2/q}\right),
\end{align*}
which is~\eqref{eq:positivity_exp}.
\end{proof}

\subsection{The general case $\sigma_1\geq0$}
\label{subsec:general_positivity}

For $\sigma_1>0$ the explicit formula~\eqref{eq:X_linear_explicit} is
no longer available, and we use the additive form of
Lemma~\ref{lem:lamperti}.

\begin{theorem}
\label{thm:positivity_general}
Assume~\eqref{eq:target_positivity}, let $\sigma_1\geq0$ and
$m_0:=\min(x_0,b_*)$. Then, up to a $\mathbb P$-null set,
\begin{equation}
\label{eq:general_inclusion}
\bigl\{\exists\,t\in[0,T]:\ X_t\le0\bigr\}
\;\subseteq\;
\bigl\{2\,C_T\ge \Lambda(m_0)\bigr\},
\end{equation}
and consequently:
\begin{enumerate}[label=(\roman*)]
\item for every $p\ge1$,
\begin{equation}
\label{eq:general_quant}
\PP\Bigl(X_t>0\ \text{for all } t\in[0,T]\Bigr)
\;\ge\;1-\left(\frac{2\,T^{H}}{\Lambda(m_0)}\right)^{p}\E\bigl[C_1^{\,p}\bigr];
\end{equation}
\item with the constant $\kappa$ of Lemma~\ref{lem:moment_growth},
\begin{equation}
\label{eq:general_exp}
\PP\Bigl(X_t>0\ \text{for all } t\in[0,T]\Bigr)
\;\ge\;1-e^{2}\exp\!\left(-\left(\frac{\Lambda(m_0)}{2e\,\kappa\,T^{H}}\right)^{2/q}\right).
\end{equation}
\end{enumerate}
Moreover,
\begin{equation}
\label{eq:F_lower}
\Lambda(m_0)\;\ge\;\frac{m_0}{\sigma_0+\sigma_1\sqrt{m_0+\eps}} .
\end{equation}
\end{theorem}

\begin{proof}
\emph{Step~1: Reduction to the additive equation.}
Let $Y=\Lambda(X)$ be as in Lemma~\ref{lem:lamperti}. We work on the event
of probability one on which~\eqref{eq:lamperti_Y} holds for every
$t\in[0,T]$ and the paths of $X$, $Y$ and $Z^{(q,H)}$ are continuous,
and we fix a trajectory in this event such that $X_t\le0$ for some
$t\in[0,T]$. We show that $2C_T\ge \Lambda(m_0)$ for this trajectory, which
proves~\eqref{eq:general_inclusion}.

\medskip
\noindent\emph{Step~2: First hitting time of $0$.}
Since $X_0=x_0>0$, $X$ is continuous and $X_t\le0$ for some $t\le T$,
the time
\[
\tau:=\inf\{t\in[0,T]:\ X_t=0\}
\]
is well defined, $\tau\in(0,T]$, $X_\tau=0$ and $X_s>0$ for
$s\in[0,\tau)$. By Lemma~\ref{lem:lamperti}, $Y_\tau=\Lambda(0)=0$ and
$Y_s>0$ for $s\in[0,\tau)$.

\medskip
\noindent\emph{Step~3: Last passage above the level $\Lambda(b_*)$.}
Set
\[
\vartheta:=\sup\bigl\{s\in[0,\tau]:\ Y_s\ge \Lambda(b_*)\bigr\},
\qquad\text{with }\vartheta:=0\text{ if this set is empty.}
\]
Since $Y$ is continuous, the set
$\{s\in[0,\tau]:\ Y_s\ge\Lambda(b_*)\}$ is closed, and it does not
contain $\tau$ because $Y_\tau=0<\Lambda(b_*)$. Hence $\vartheta<\tau$.
We claim that $Y_\vartheta\ge \Lambda(m_0)$. Two cases can occur.
\begin{itemize}
\item If the set is empty, then $\vartheta=0$ and
$Y_\vartheta=\Lambda(x_0)\ge \Lambda(m_0)$, since $\Lambda$ is
increasing and $x_0\ge m_0$.
\item If the set is non-empty, it is closed, so its supremum
$\vartheta$ belongs to it and $Y_\vartheta\ge\Lambda(b_*)\ge\Lambda(m_0)$,
since $b_*\ge m_0$.
\end{itemize}
In both cases $Y_\vartheta\ge \Lambda(m_0)>0$. Moreover, by definition
of $\vartheta$, $Y_s<\Lambda(b_*)$ for every $s\in(\vartheta,\tau]$.

\medskip
\noindent\emph{Step~4: Sign of the drift between $\vartheta$ and $\tau$.}
For $s\in(\vartheta,\tau)$ we have $0<Y_s<\Lambda(b_*)$ by Steps~2 and~3
(recall that $\vartheta<\tau$), hence
$0<X_s<b_*$ because $\Lambda$ is strictly increasing. Since
$b(s)\ge b_*$ for every $s$, it follows that
\[
\frac{a\bigl(b(s)-X_s\bigr)}{\tilde\sigma(X_s)}
\;\ge\;\frac{a\,(b_*-X_s)}{\tilde\sigma(X_s)}
\;\ge\;0,
\qquad s\in(\vartheta,\tau).
\]

\medskip
\noindent\emph{Step~5: Pathwise conclusion.}
Subtracting~\eqref{eq:lamperti_Y} at times $\vartheta$ and $\tau$, and using
Steps~3 and~4,
\begin{align*}
0=Y_\tau
&=Y_\vartheta+\int_\vartheta^\tau\frac{a\bigl(b(s)-X_s\bigr)}{\tilde\sigma(X_s)}\,ds
+\bigl(Z_\tau^{(q,H)}-Z_\vartheta^{(q,H)}\bigr)\\
&\ge \Lambda(m_0)+0+\bigl(Z_\tau^{(q,H)}-Z_\vartheta^{(q,H)}\bigr).
\end{align*}
Therefore
\[
\Lambda(m_0)\;\le\;Z_\vartheta^{(q,H)}-Z_\tau^{(q,H)}
\;\le\;\bigl|Z_\vartheta^{(q,H)}\bigr|+\bigl|Z_\tau^{(q,H)}\bigr|
\;\le\;2\,C_T,
\]
which proves~\eqref{eq:general_inclusion}.

\medskip
\noindent\emph{Step~6: Probability bounds.}
By~\eqref{eq:general_inclusion}, Markov's inequality and
Corollary~\ref{cor:CT_moments}, for every $p\ge1$,
\[
\PP\bigl(\exists\,t\le T:\ X_t\le0\bigr)
\le\PP\bigl(C_T\ge\tfrac12\Lambda(m_0)\bigr)
\le\Bigl(\frac{2}{\Lambda(m_0)}\Bigr)^{p}\E\bigl[C_T^{\,p}\bigr]
=\Bigl(\frac{2\,T^H}{\Lambda(m_0)}\Bigr)^{p}\E\bigl[C_1^{\,p}\bigr],
\]
which is~\eqref{eq:general_quant}. Since $C_T$ has the law of $T^HC_1$,
\eqref{eq:C1_tail} applied with $x=\Lambda(m_0)/(2T^H)$ gives
\[
\PP\bigl(\exists\,t\le T:\ X_t\le0\bigr)
\le\PP\Bigl(C_1\ge\frac{\Lambda(m_0)}{2T^H}\Bigr)
\le e^{2}\exp\!\left(-\left(\frac{\Lambda(m_0)}{2e\,\kappa\,T^{H}}\right)^{2/q}\right),
\]
which is~\eqref{eq:general_exp}.

\medskip
\noindent\emph{Step~7: Lower bound on $\Lambda(m_0)$.}
For $u\in[0,m_0]$, condition~(a) on $\phi_\eps$ gives
$\phi_\eps(u)=u+\eps\le m_0+\eps$, hence
$\tilde\sigma(u)\le\sigma_0+\sigma_1\sqrt{m_0+\eps}$ and
\[
\Lambda(m_0)=\int_0^{m_0}\frac{du}{\tilde\sigma(u)}
\ge\int_0^{m_0}\frac{du}{\sigma_0+\sigma_1\sqrt{m_0+\eps}}
=\frac{m_0}{\sigma_0+\sigma_1\sqrt{m_0+\eps}},
\]
which is~\eqref{eq:F_lower}.
\end{proof}

\begin{remark}[Comparison with Theorem~\ref{thm:positivity}]
\label{rem:general_vs_linear}
For $\sigma_1=0$ we have $\Lambda(x)=x/\sigma_0$, and
Theorem~\ref{thm:positivity_general} gives the threshold
$m_0/(2\sigma_0T^H)$ for $C_1$, whereas Theorem~\ref{thm:positivity}
gives $m_0/((2-e^{-aT})\sigma_0T^H)$. Since $1<2-e^{-aT}<2$, the latter
is sharper, by a factor close to $2$ when $aT$ is small, which
multiplies the exponent in~\eqref{eq:positivity_exp} by up to
$2^{2/q}$. The reason is that the explicit
formula~\eqref{eq:X_linear_explicit} bounds the noise contribution by
$|Z_t^{(q,H)}|$ plus a weighted integral of $|Z_s^{(q,H)}|$, that is by
$(2-e^{-aT})C_T$, whereas the argument of
Theorem~\ref{thm:positivity_general} bounds it by the oscillation
$Z_\vartheta^{(q,H)}-Z_\tau^{(q,H)}\le2C_T$. That argument only uses
the sign of the drift in the band $\{0<x<b_*\}$, so the speed $a$ does
not enter the bound. In both theorems, the probability of leaving
$(0,\infty)$ is controlled from above through the ratio between a
distance to the boundary, measured in the natural scale $\Lambda$ of
the diffusion coefficient, and the size $T^H$ of the noise.
\end{remark}

\begin{corollary}[Dependence on $m_0/\sigma_1^2$]
\label{cor:feller_ratio}
Under the assumptions of Theorem~\ref{thm:positivity_general}, assume
$\sigma_1>0$ and
\begin{equation}
\label{eq:feller_regime}
\sigma_0+\sigma_1\sqrt{\eps}\;\le\;\sigma_1\sqrt{m_0} .
\end{equation}
Then
\[
\PP\bigl(\exists\,t\le T:\ X_t\le0\bigr)
\;\le\;e^{2}\exp\!\left(-\left(\frac{m_0}{16\,e^{2}\kappa^{2}\,\sigma_1^{2}\,T^{2H}}\right)^{1/q}\right).
\]
\end{corollary}

\begin{proof}
Since $\sqrt{m_0+\eps}\le\sqrt{m_0}+\sqrt{\eps}$,
\eqref{eq:F_lower} and~\eqref{eq:feller_regime} give
\[
\Lambda(m_0)\;\ge\;\frac{m_0}{\sigma_0+\sigma_1\sqrt{\eps}+\sigma_1\sqrt{m_0}}
\;\ge\;\frac{m_0}{2\sigma_1\sqrt{m_0}}
\;=\;\frac{\sqrt{m_0}}{2\sigma_1}.
\]
Hence
\[
\left(\frac{\Lambda(m_0)}{2e\,\kappa\,T^{H}}\right)^{2/q}
\;\ge\;\left(\frac{\sqrt{m_0}}{4e\,\kappa\,\sigma_1T^{H}}\right)^{2/q}
\;=\;\left(\frac{m_0}{16\,e^{2}\kappa^{2}\,\sigma_1^{2}\,T^{2H}}\right)^{1/q},
\]
and the claim follows from~\eqref{eq:general_exp}.
\end{proof}

\begin{remark}[On the constants and the scope of the bounds]
\label{rem:feller_ratio}
\begin{enumerate}[label=(\alph*)]
\item On $[0,\infty)$ the function $\Lambda$ is explicit: the
substitution $v=\sqrt{u+\eps}$ gives, for $x\ge0$,
\[
\Lambda(x)=\int_{\sqrt\eps}^{\sqrt{x+\eps}}\frac{2v\,dv}{\sigma_0+\sigma_1v}
=\frac{2}{\sigma_1}\bigl(\sqrt{x+\eps}-\sqrt\eps\bigr)
-\frac{2\sigma_0}{\sigma_1^2}\log\frac{\sigma_0+\sigma_1\sqrt{x+\eps}}{\sigma_0+\sigma_1\sqrt\eps}
\]
when $\sigma_1>0$. The bound~\eqref{eq:general_exp} can therefore be
evaluated exactly, and~\eqref{eq:F_lower} is only a convenient lower
bound. For instance, as $\sigma_0\to0$ and $\eps\to0$,
$\Lambda(m_0)\to2\sqrt{m_0}/\sigma_1$, four times the lower bound used
in the proof of Corollary~\ref{cor:feller_ratio}; in this regime the
constant $16$ in the corollary can be replaced by any constant larger
than $1$.
\item The ratio $m_0/\sigma_1^2$ in Corollary~\ref{cor:feller_ratio} is
the natural scale of a square-root diffusion coefficient, but the
corollary is not an analogue of the Feller condition: it is a
quantitative bound and not an almost-sure statement, and the
mean-reversion speed $a$ does not appear in it
(Remark~\ref{rem:general_vs_linear}). Condition~\eqref{eq:feller_regime}
requires the square-root part of the volatility at level $m_0$ to
dominate $\sigma_0+\sigma_1\sqrt\eps=\tilde\sigma(0)$; it implies
$m_0>\eps$.
\item All the bounds of this section depend on the noise only through
$C_T$, whose law is that of $T^HC_1$. They deteriorate as $T$ grows
and they ignore the restoring effect of the drift, which in reality
pushes $X$ back towards $b$. They are informative when the distance
$\Lambda(m_0)$ to the boundary is large compared with $T^H\kappa$, and
they are not informative otherwise. For parameter values of the size
found in interest-rate applications, the term $e^{2}\exp(\cdots)$
in~\eqref{eq:general_exp} can exceed $1$, so that the bound is trivial; Section~\ref{sec:simulations}
compares it with simulated frequencies.
\item The exponent $2/q$ in~\eqref{eq:positivity_exp}
and~\eqref{eq:general_exp} is that of the tail of $C_1$, which cannot
be improved for $q\le2$ (Remark~\ref{rem:tail_sharp}). Whether the
probability of reaching $0$ itself has a tail of this order is a
different question that we do not address.
\end{enumerate}
\end{remark}

\subsection{Almost-sure positivity}
\label{subsec:no_as}

In the Brownian CIR model~\eqref{eq:classical_cir}, positivity under
$2ab\geq\sigma^2$ is proved by applying It\^o's formula to $\log X_t$ or
$1/X_t$: the second-order term is what keeps the process away from $0$,
and it exists because the diffusion coefficient vanishes at $0$. Neither
ingredient is present here. Young calculus follows the ordinary chain
rule, and the diffusion coefficient does not vanish at the boundary:
$\tilde\sigma(0)=\sigma_0+\sigma_1\sqrt\eps>0$. The next proposition
shows that, for $q\le2$ and $\sigma_1=0$, the probabilities of
Theorem~\ref{thm:positivity} are indeed strictly smaller than $1$.

\begin{proposition}[No almost-sure positivity for $q\le2$, $\sigma_1=0$]
\label{prop:no_as}
Let $q\in\{1,2\}$, $\sigma_1=0$ and let $x_0,a,\sigma_0>0$ and $b$ be as
in Section~\ref{sec:model}. Then $\PP(X_t<0)>0$ for every $t\in(0,T]$.
\end{proposition}

\begin{proof}
Fix $t\in(0,T]$. By~\eqref{eq:X_linear_explicit}, $X_t=e^{-at}x_0+M_t+I_t$
with $M_t$ deterministic, so it suffices to show that
$\PP(I_t<-M)>0$ for every $M>0$.

\emph{Step~1: the kernel of $I_t$.} The map $s\mapsto L_s^{H,q}$ is
continuous from $[0,T]$ to $L^2(\R^q)$ by
Proposition~\ref{prop:hermite_properties}(vi), and $I_q$ is an isometry up
to the factor $\sqrt{q!}$. Hence the Riemann sums of
$\int_0^te^{-a(t-s)}Z_s^{(q,H)}\,ds$ converge in $L^2(\Omega)$ to
$I_q\bigl(\int_0^te^{-a(t-s)}L_s^{H,q}\,ds\bigr)$, and they converge
pathwise to the integral, so $I_t=\sigma_0\,I_q(f_t)$ with
$f_t:=L_t^{H,q}-a\int_0^te^{-a(t-s)}L_s^{H,q}\,ds$. Writing
$\varphi_u(\xi):=(u-\xi)_+^{H_0-3/2}$ and using~\eqref{eq:hermite_def},
Tonelli's theorem (all integrands are non-negative) gives
\begin{align*}
f_t(\xi_1,\dots,\xi_q)
&=c(H,q)\Bigl(\int_0^t\prod_{j=1}^q\varphi_u(\xi_j)\,du
-a\int_0^te^{-a(t-s)}\int_0^s\prod_{j=1}^q\varphi_u(\xi_j)\,du\,ds\Bigr)\\
&=c(H,q)\int_0^t\Bigl(1-a\int_u^te^{-a(t-s)}\,ds\Bigr)\prod_{j=1}^q\varphi_u(\xi_j)\,du\\
&=c(H,q)\int_0^te^{-a(t-u)}\prod_{j=1}^q\varphi_u(\xi_j)\,du .
\end{align*}

\emph{Step~2: $q=1$.} Then $f_t$ is a non-negative function on $\R$,
positive on $(-\infty,t)$, so $\|f_t\|_{L^2(\R)}>0$ and
$I_t\sim\mathcal N(0,\sigma_0^2\|f_t\|^2_{L^2(\R)})$ is a non-degenerate
centred Gaussian variable. Hence $\PP(I_t<-M)>0$ for every $M$.

\emph{Step~3: $q=2$.} Let $\mathcal T$ be the integral operator on
$L^2(\R)$ with kernel $f_t$. Since $f_t$ is symmetric and belongs to
$L^2(\R^2)$, $\mathcal T$ is a self-adjoint Hilbert--Schmidt operator, so
that $f_t=\sum_k\lambda_ke_k\otimes e_k$ in $L^2(\R^2)$, where $(e_k)$
is an orthonormal family and $\lambda_k\ne0$ are the non-zero
eigenvalues, with $\sum_k\lambda_k^2<\infty$. Let
$\mu(du):=c(H,2)e^{-a(t-u)}\mathbf 1_{[0,t]}(u)\,du$, so that
$f_t(\xi_1,\xi_2)=\int\varphi_u(\xi_1)\varphi_u(\xi_2)\,\mu(du)$ by
Step~1. For $g$ bounded with compact support, Fubini's theorem (justified
because $f_t\ge0$ and $\int\!\!\int f_t\,|g|\otimes|g|<\infty$) gives
\[
\langle\mathcal Tg,g\rangle_{L^2(\R)}=\int\Bigl(\int_\R\varphi_u(\xi)g(\xi)\,d\xi\Bigr)^2\mu(du)\ge0 ,
\]
so $\mathcal T$ is non-negative by density, and $\lambda_k>0$ for every $k$.

We show that $\sum_k\lambda_k=\infty$. For every finite orthonormal family
$(h_j)$ in $L^2(\R)$, Bessel's inequality gives
\[
\sum_j\langle\mathcal Th_j,h_j\rangle
=\sum_k\lambda_k\sum_j\langle e_k,h_j\rangle^2\le\sum_k\lambda_k .
\]
Here $H_0=\frac{1+H}2$, and we write $\gamma:=H_0-\frac32=\frac H2-1\in(-\frac34,-\frac12)$,
so that $2\gamma+1=H-1$. Let $\frac t4\le\xi<\xi'\le\frac t2$ and
$d:=\xi'-\xi\le\frac t4$. For $u\in(\xi',\xi'+d)\subset(0,t)$ we have
$0<u-\xi\le2d$, hence $(u-\xi)^\gamma\ge(2d)^\gamma$, and
\[
f_t(\xi,\xi')\ge c(H,2)\,e^{-at}\int_{\xi'}^{\xi'+d}(2d)^\gamma(u-\xi')^\gamma\,du
=\frac{c(H,2)\,e^{-at}\,2^\gamma}{\gamma+1}\,d^{2\gamma+1}
=:c_1\,|\xi-\xi'|^{H-1}.
\]
For $n\ge1$, let $\ell:=t/(4n)$, $I_j:=[\frac t4+(j-1)\ell,\frac t4+j\ell]$
and $h_j:=\ell^{-1/2}\mathbf 1_{I_j}$, $j=1,\dots,n$, an orthonormal
family. Since $\int_0^\ell\!\int_0^\ell|x-y|^{H-1}\,dx\,dy=2\ell^{H+1}/(H(H+1))$,
\[
\sum_{j=1}^n\langle\mathcal Th_j,h_j\rangle
=\sum_{j=1}^n\frac1\ell\int_{I_j}\!\int_{I_j}f_t
\ge n\,\frac{2c_1\,\ell^{H}}{H(H+1)}
=\frac{c_1\,t}{2H(H+1)}\,\ell^{H-1}\;\longrightarrow\;\infty
\qquad(n\to\infty),
\]
because $H<1$. Hence $\sum_k\lambda_k=\infty$.

Finally, $I_2(e_k\otimes e_k)=\xi_k^2-1$ with $\xi_k:=I_1(e_k)$, and the
$\xi_k$ are i.i.d.\ $\mathcal N(0,1)$, so that, by continuity of $I_2$,
\[
I_t=\sigma_0\sum_k\lambda_k(\xi_k^2-1)
\qquad\text{in }L^2(\Omega).
\]
Fix $M>0$ and choose $N$ with $\frac12\sum_{k\le N}\lambda_k>M/\sigma_0$.
The variable $R_N:=\sum_{k>N}\lambda_k(\xi_k^2-1)$ is centred and
independent of $(\xi_1,\dots,\xi_N)$, so $\PP(R_N\le0)>0$. On the event
$\{\xi_k^2<\frac12\text{ for }k\le N\}\cap\{R_N\le0\}$, which has
positive probability,
\[
I_t\le\sigma_0\sum_{k\le N}\lambda_k\Bigl(\frac12-1\Bigr)<-M .
\]
\end{proof}

\begin{remark}[Comparison with the square-root model]
\label{rem:no_as}
\begin{enumerate}[label=(\alph*)]
\item Proposition~\ref{prop:no_as} is a consequence of
$\tilde\sigma(0)>0$, not of the particular form of $\tilde\sigma$: the
same mechanism operates for $\sigma_1>0$, since near $0$ the equation
behaves like one with additive noise $\tilde\sigma(0)\,dZ^{(q,H)}$. For
fractional Brownian drivers, the behaviour at zero of fractional
Cox--Ingersoll--Ross type processes, including the probability of
reaching zero, is also studied by Mishura, Piterbarg, Ralchenko and
Yurchenko-Tytarenko~\cite{MPRYT2018}.
\item For the pure square-root model ($\sigma_0=0$, $\eps=0$, constant
target $b>0$), the situation is the opposite. Mishura and
Yurchenko-Tytarenko~\cite{mishura2017} show, by a pathwise argument that
uses only the H\"older continuity of fBm with $H>1/2$, that the
fractional CIR process is strictly positive almost surely for every
volatility level; the last-passage argument in the proof of
Theorem~\ref{thm:positivity_general} is an adaptation of theirs. Di~Nunno,
Mishura and Yurchenko-Tytarenko~\cite{dinunno2023} carry out this argument
for an arbitrary H\"older continuous driver; for a driver of H\"older
order larger than $1/2$ it yields CIR-type equations
$dX_t=(\kappa_0-\theta X_t)\,dt+\nu\sqrt{X_t}\,dZ_t$, with
$\kappa_0,\theta,\nu>0$ and a pathwise Young integral, which have a
unique solution in the class of non-negative H\"older continuous
processes, and this solution never reaches
$0$~\cite[Section~5.1 and Remark~5.2]{dinunno2023}. Since $Z^{(q,H)}$ is
H\"older continuous of every order smaller than $H$, with a H\"older
constant having moments of all orders (Lemma~\ref{lem:grr}), the
Hermite-driven square-root model is strictly positive almost surely for
every $q\geq1$ and every $H\in(\frac12,1)$. After the change of variables
$Y=\sqrt X$, their argument relies on a drift that is singular at the
boundary. When $\tilde\sigma(0)>0$, the drift $k$ of the transformed
equation~\eqref{eq:lamperti_Y} is globally Lipschitz
(Lemma~\ref{lem:drift_bounds}), there is no such singularity, and a path
of $Z^{(q,H)}$ that decreases fast enough drives $X$ below $0$.
\item The argument of Proposition~\ref{prop:no_as} uses the explicit
eigenvalue structure of the second chaos. For $q\ge3$ and $\sigma_1=0$,
and for every $q$ when $\sigma_1>0$, we expect the same conclusion but do
not prove it.
\end{enumerate}
\end{remark}

\begin{remark}[Negative long-run level]
\label{rem:ZLB}
Condition~\eqref{eq:target_positivity} requires $b_0>b_1\geq0$. If a
fitted level $b_0$ is non-positive, the model is still well posed, and
the results of Section~\ref{sec:malliavin} still apply, but the
positivity bounds do not.
\end{remark}

\section{Malliavin differentiability and density}
\label{sec:malliavin}

The general theory of Loosveldt, Nachit and Nourdin~\cite{loosveldt2025}
gives Malliavin differentiability and densities for equations driven by
elements of a Wiener chaos, under the assumption that the coefficients
are bounded with bounded derivatives. In~\eqref{eq:main_sde} the drift
grows linearly and $\tilde\sigma$ grows like a square root, and the drift
depends on time. We therefore give a direct proof, which relies on the
additive form of Lemma~\ref{lem:lamperti}. In this form, $Y_t=\Lambda(X_t)$
is a Lipschitz and differentiable function of the path of $Z^{(q,H)}$,
and the integrability of the derivative follows from
Lemma~\ref{lem:drift_bounds}. For $q\geq2$, the generic a priori bounds
for Young equations, of order $\exp(C\,K^{1/\gamma})$ with $K$ a H\"older
constant of the driver and $\gamma<1$, need not have moments, since the
tails of $Z^{(q,H)}$ are not sub-Gaussian when $q\ge2$
(Remark~\ref{rem:tail_sharp} for $q=2$); the one-dimensional change of
variables avoids this difficulty.

\subsection{The solution map}

Fix $t\in(0,T]$. For $z\in C([0,T])$, let $\Phi(z)=(\Phi_s(z))_{s\in[0,T]}$
be the solution of
\begin{equation}
\label{eq:solution_map}
y_s=\Lambda(x_0)+\int_0^sk(r,y_r)\,dr+z_s,
\qquad s\in[0,T],
\end{equation}
with $k$ defined in~\eqref{eq:g_h}. By Lemma~\ref{lem:drift_bounds}(ii),
$k$ is continuous, continuously differentiable in $y$, with
$\partial_yk(r,y)=-a(1+g(r,\Lambda^{-1}(y)))$ and
$|\partial_yk|\le L:=a(1+G)$, so~\eqref{eq:solution_map} has a unique
continuous solution for every $z$. By Lemma~\ref{lem:lamperti},
$Y_s=\Lambda(X_s)=\Phi_s(Z^{(q,H)})$ for all $s$, almost surely.

\begin{lemma}
\label{lem:solution_map}
\begin{enumerate}[label=(\roman*)]
\item For $z,z'\in C([0,T])$,
$\sup_{s\le T}|\Phi_s(z)-\Phi_s(z')|\le e^{LT}\sup_{s\le T}|z_s-z'_s|$.
\item The map $z\mapsto\Phi_t(z)$ is continuously Fr\'echet differentiable
on $C([0,T])$ with the uniform norm, and, with $y=\Phi(z)$,
\begin{equation}
\label{eq:Phi_derivative}
D\Phi_t(z)[w]=w_t-\int_0^t\rho^z_t(s)\,w_s\,ds,
\qquad
\rho^z_t(s):=a\bigl(1+\gamma^z_s\bigr)\exp\Bigl(-a\int_s^t\bigl(1+\gamma^z_r\bigr)\,dr\Bigr),
\end{equation}
where $\gamma^z_r:=g(r,\Lambda^{-1}(y_r))$. Moreover
$|\rho^z_t(s)|\le R:=a(1+G)e^{a(G-1)^+T}$, with $\rho^z_t\ge0$ when
$G\le1$, so that $|D\Phi_t(z)[w]|\le(1+RT)\sup_{s\le t}|w_s|$.
\end{enumerate}
\end{lemma}

\begin{proof}
(i) With $y=\Phi(z)$ and $y'=\Phi(z')$,
$|y_s-y'_s|\le\int_0^sL|y_r-y'_r|\,dr+\sup_{r\le T}|z_r-z'_r|$, and
Gronwall's lemma gives the claim.

(ii) Let $w\in C([0,T])$ and let $v$ be the solution of the linear
equation $v_s=\int_0^s\partial_yk(r,y_r)\,v_r\,dr+w_s$. Writing
$\ell_r:=\partial_yk(r,y_r)=-a(1+\gamma^z_r)$, the function $v-w$ solves
$(v-w)'_s=\ell_s(v-w)_s+\ell_sw_s$ with $(v-w)_0=0$, hence
\[
v_t=w_t+\int_0^t\ell_s\,e^{\int_s^t\ell_r\,dr}\,w_s\,ds
=w_t-\int_0^t\rho^z_t(s)\,w_s\,ds .
\]
Let $y^\eta:=\Phi(z+\eta w)$ for $\eta\in(0,1]$. Then
\[
y^\eta_s-y_s-\eta v_s
=\int_0^s\bigl(k(r,y^\eta_r)-k(r,y_r)-\ell_r(y^\eta_r-y_r)\bigr)\,dr
+\int_0^s\ell_r\bigl(y^\eta_r-y_r-\eta v_r\bigr)\,dr .
\]
By (i), $\sup_r|y^\eta_r-y_r|\le e^{LT}\eta\|w\|_\infty$. Since
$\partial_yk$ is uniformly continuous on
$[0,T]\times[-M,M]$, where $M$ bounds $|y|$ and $|y^\eta|$, the first
integral is $o(\eta\|w\|_\infty)$, uniformly in $s$ and in $w$ with
$\|w\|_\infty\le1$; Gronwall's lemma then gives
$\sup_s|y^\eta_s-y_s-\eta v_s|=o(\eta)$ uniformly in such $w$. Hence
$\Phi_t$ is Fr\'echet differentiable with derivative $w\mapsto v_t$.
Continuity of $z\mapsto D\Phi_t(z)$ follows from (i) and the continuity of
$g$. Finally, $|\gamma^z_r|\le G$ by Lemma~\ref{lem:drift_bounds}(i), so
that $0\le1+\gamma^z_r\le1+G$ when $G\le1$, $|1+\gamma^z_r|\le1+G$ in all
cases, and $-a(1+\gamma^z_r)\le a(G-1)\le a(G-1)^+$. Hence
$\exp(-a\int_s^t(1+\gamma^z_r)dr)\le e^{a(G-1)^+T}$ and
$|\rho^z_t(s)|\le a(1+G)e^{a(G-1)^+T}$.
\end{proof}

\subsection{Malliavin differentiability}

\begin{theorem}
\label{thm:malliavin}
For every $t\in(0,T]$, $Y_t$ and $X_t$ belong to $\D^{1,\infty}$, and
\begin{equation}
\label{eq:malliavin_derivative}
DY_t=DZ_t^{(q,H)}-\int_0^t\rho_t(s)\,DZ_s^{(q,H)}\,ds,
\qquad
DX_t=\tilde\sigma(X_t)\,DY_t,
\end{equation}
where $\rho_t(s):=\rho^{Z}_t(s)=a\bigl(1+g(s,X_s)\bigr)\exp\bigl(-a\int_s^t(1+g(r,X_r))\,dr\bigr)$,
$|\rho_t(s)|\le R$, and the integral is a Bochner integral in $\fH$.
\end{theorem}

\begin{proof}
\emph{Step~1: approximation.} For $n\geq1$, let $t_i:=iT/n$ and let
$Z^{(n)}$ be the piecewise linear interpolation of $Z^{(q,H)}$ at the
points $t_0,\dots,t_n$. Then $\Phi_t(Z^{(n)})=\varphi_n(Z_{t_0}^{(q,H)},\dots,Z_{t_n}^{(q,H)})$,
where $\varphi_n(x):=\Phi_t(\iota_nx)$ and $\iota_n:\R^{n+1}\to C([0,T])$
is the linear interpolation map, which satisfies
$\|\iota_nx\|_\infty=\max_i|x_i|$. By Lemma~\ref{lem:solution_map}(ii),
$\varphi_n\in C^1(\R^{n+1})$ with
$\sum_i|\partial_i\varphi_n(x)|\le1+RT$. Since each $Z^{(q,H)}_{t_i}$
belongs to $\D^{1,\infty}$, the chain rule (M1) gives
$\Phi_t(Z^{(n)})\in\D^{1,p}$ for every $p$, with
\[
D\,\Phi_t(Z^{(n)})=\sum_{i=0}^n\partial_i\varphi_n\bigl(Z^{(q,H)}_{t_0},\dots,Z^{(q,H)}_{t_n}\bigr)\,DZ_{t_i}^{(q,H)}
=D\Phi_t(Z^{(n)})\bigl[(DZ^{(q,H)})^{(n)}\bigr],
\]
where $(DZ^{(q,H)})^{(n)}$ is the piecewise linear interpolation of the
$\fH$-valued path $s\mapsto DZ_s^{(q,H)}$ and $D\Phi_t(z)$ is extended to
$\fH$-valued paths by~\eqref{eq:Phi_derivative}. Hence
\[
\bigl\|D\,\Phi_t(Z^{(n)})\bigr\|_\fH\le(1+RT)\sup_{s\le T}\|DZ_s^{(q,H)}\|_\fH ,
\]
whose $p$-th moment is finite and independent of $n$ by
Lemma~\ref{lem:DZ}(ii).

\emph{Step~2: passage to the limit.} Since the paths of $Z^{(q,H)}$ are
continuous, $\sup_s|Z^{(n)}_s-Z^{(q,H)}_s|\to0$ almost surely, and by
Lemma~\ref{lem:solution_map}(i),
$|\Phi_t(Z^{(n)})-Y_t|\le e^{LT}\sup_s|Z^{(n)}_s-Z^{(q,H)}_s|\le2e^{LT}C_T$.
By dominated convergence and Corollary~\ref{cor:CT_moments},
$\Phi_t(Z^{(n)})\to Y_t$ in $L^p(\Omega)$ for every $p$. By (M2),
$Y_t\in\D^{1,p}$ for every $p>1$. Moreover, $(DZ^{(q,H)})^{(n)}\to DZ^{(q,H)}$
uniformly on $[0,T]$ in $\fH$ by Lemma~\ref{lem:DZ}(ii). Moreover
$\Phi(Z^{(n)})\to\Phi(Z^{(q,H)})$ uniformly on $[0,T]$ by
Lemma~\ref{lem:solution_map}(i), and $(r,y)\mapsto g(r,\Lambda^{-1}(y))$
is uniformly continuous in $y$ on compact sets, uniformly in $r$; hence
$\gamma^{Z^{(n)}}\to\gamma^{Z}$ and $\rho^{Z^{(n)}}_t\to\rho_t$ uniformly
on $[0,t]$, so that $D\,\Phi_t(Z^{(n)})$ converges almost
surely, and by dominated convergence in $L^p(\Omega;\fH)$, to the
right-hand side of the first identity in~\eqref{eq:malliavin_derivative}.
This limit is therefore $DY_t$.

\emph{Step~3: from $Y_t$ to $X_t$.} We have $X_t=\psi(Y_t)$ with
$\psi:=\Lambda^{-1}$, which is $C^1$ with
$\psi'(y)=\tilde\sigma(\psi(y))$. By Lemmas~\ref{lem:sigma}
and~\ref{lem:drift_bounds}(iii),
$|\psi'(y)|\le\sigma_0+\sigma_1\sqrt{|\psi(y)|+\eps}\le C(1+|y|)$ for a
constant $C$. Let $\chi_N\in C^1(\R)$ with $\chi_N(y)=y$ for $|y|\le N$,
$|\chi_N(y)|\le\min(|y|,N+1)$ and $0\le\chi_N'\le1$. The function
$\psi\circ\chi_N$ has a bounded derivative, at most
$\sup_{|z|\le N+1}|\psi'(z)|$, so by (M1),
$\psi(\chi_N(Y_t))\in\D^{1,p}$ with
$D\psi(\chi_N(Y_t))=\psi'(\chi_N(Y_t))\chi_N'(Y_t)DY_t$. As $N\to\infty$,
$\psi(\chi_N(Y_t))\to X_t$ almost surely, and both
$|\psi(\chi_N(Y_t))|\le|\psi(0)|+C(|Y_t|+Y_t^2)$ and
$\|D\psi(\chi_N(Y_t))\|_\fH\le C(1+|Y_t|)\|DY_t\|_\fH$ have moments of all
orders, uniformly in $N$, by the proof of Lemma~\ref{lem:X_moments}.
By dominated convergence and (M2), $X_t\in\D^{1,p}$ for every $p>1$ and
$DX_t=\psi'(Y_t)DY_t=\tilde\sigma(X_t)DY_t$.
\end{proof}

\begin{remark}
\label{rem:malliavin_remarks}
Integrating by parts in~\eqref{eq:malliavin_derivative}, using
$DZ_0^{(q,H)}=0$, one obtains the representation
$DX_t=\int_0^t\Theta_t(s)\,dDZ_s^{(q,H)}$, as an $\fH$-valued Young
integral, with the kernel
\[
\Theta_t(s)=\tilde\sigma(X_t)\exp\Bigl(-a\int_s^t\bigl(1+g(r,X_r)\bigr)\,dr\Bigr),
\qquad0\le s\le t,
\]
which is the form of the derivative obtained in~\cite{loosveldt2025} for
equations with bounded coefficients; here it is explicit, and
$\sigma_0e^{-a(1+G)T}\le\Theta_t(s)\le\tilde\sigma(X_t)e^{a(G-1)^+T}$. For
$q=1$, Malliavin differentiability and the existence of densities for
equations driven by fBm with $H>1/2$ are
classical~\cite{nualart2009malliavin,HuNualart2007,BaudoinHairer2007}; the point of
Theorem~\ref{thm:malliavin} is the case $q\geq2$, for which the
integrability of the derivative is not a consequence of general a priori
bounds. The square root without regularisation ($\eps=0$) is not covered.
\end{remark}

\subsection{Absolute continuity of the law}

\begin{theorem}
\label{thm:density}
For every $t\in(0,T]$, the law of $X_t$ is absolutely continuous with
respect to Lebesgue measure on $\R$.
\end{theorem}

\begin{proof}
By Theorem~\ref{thm:malliavin} and (M3), it suffices to show that
$\|DX_t\|_\fH>0$ almost surely, and since $\tilde\sigma\geq\sigma_0>0$,
that $\|DY_t\|_\fH>0$ almost surely. Let $\Omega_0:=\{DY_t=0\}$, defined
from the version of $DY_t$ given by~\eqref{eq:malliavin_derivative}, and
assume, to reach a contradiction, that $\PP(\Omega_0)>0$.

Fix $\theta\in(0,H)$ and $\delta\in(0,t]$, and let $P_\delta$ be as in
Lemma~\ref{lem:scaling}. By~\eqref{eq:DZ_adapted}, $P_\delta DZ_s^{(q,H)}=0$
for $s\le t-\delta$, and
$P_\delta DZ_s^{(q,H)}=P_\delta\bigl(DZ_s^{(q,H)}-DZ_{t-\delta}^{(q,H)}\bigr)$ for
$s\in[t-\delta,t]$. Applying $P_\delta$ to~\eqref{eq:malliavin_derivative}
and using $|\rho_t(s)|\le R$ and Lemma~\ref{lem:DZ}(ii), we get
\[
\begin{aligned}
\|P_\delta DY_t\|_\fH
&\geq\|P_\delta DZ_t^{(q,H)}\|_\fH-\int_{t-\delta}^t|\rho_t(s)|\,
\bigl\|DZ_s^{(q,H)}-DZ_{t-\delta}^{(q,H)}\bigr\|_\fH\,ds\\
&\geq\|P_\delta DZ_t^{(q,H)}\|_\fH-R\,K_\theta^D\,\delta^{1+\theta}.
\end{aligned}
\]
On $\Omega_0$ the left-hand side vanishes, so that, with
$N_\delta:=\delta^{-H}\|P_\delta DZ_t^{(q,H)}\|_\fH$,
\[
N_\delta\le R\,K_\theta^D\,\delta^{1+\theta-H}\qquad\text{on }\Omega_0 .
\]
Since $1+\theta-H>0$ and $K^D_\theta<\infty$ almost surely, $N_{1/n}\to0$
almost surely on $\Omega_0$ as $n\to\infty$. Let $\eta>0$. By
Lemma~\ref{lem:scaling}, $N_{1/n}$ has the law of $N$ for every $n\geq1/t$,
so that, by dominated convergence,
\[
\PP(N\le\eta)=\PP(N_{1/n}\le\eta)\geq\PP\bigl(\Omega_0\cap\{N_{1/n}\le\eta\}\bigr)
\;\longrightarrow\;\PP(\Omega_0),
\qquad n\to\infty .
\]
Hence $\PP(N\le\eta)\geq\PP(\Omega_0)$ for every $\eta>0$, and
$\PP(N=0)\geq\PP(\Omega_0)>0$, which contradicts Lemma~\ref{lem:scaling}.
\end{proof}

\begin{remark}
\label{rem:density_remarks}
The proof only uses that $\tilde\sigma$ is bounded below by a positive
constant, which holds as soon as $\sigma_0>0$ (and also when $\sigma_1>0$,
since $\phi_\eps\geq m_\eps>0$). For $q=1$ and $\sigma_1=0$, $X_t$ is
Gaussian with positive variance, and its density is positive on $\R$. In
general, since $\tilde\sigma>0$ everywhere, nothing in the equation
confines $X_t$ to $(0,\infty)$; we do not determine the support of the
density.
\end{remark}

\section{Numerical illustrations}
\label{sec:simulations}

\subsection{Simulation method}
\label{subsec:sim_method}

We simulate $Z^{(q,H)}$ with the wavelet-type algorithm of Ayache,
Hamonier and Loosveldt~\cite{AyacheNum2025}, which applies to Hermite
processes of every order. It approximates $Z^{(q,H)}$ by a truncated
series
\[
\widetilde S^{(q)}_{J}(t)=2^{-JH}\sum_{k\in\mathcal J^1_J(t)}\sigma_{J,k}
\int_\R\prod_{\ell=1}^q\Phi_\Delta^{(H_0-1/2)}(s-k_\ell)\,ds,
\qquad H_0=1+\frac{H-1}q,
\]
where $J$ is the resolution level, $\Phi_\Delta^{(\delta)}$ is the
fractional scaling function of Meyer~\cite{Ayache2025}, the coefficients
$\sigma_{J,k}$ are computed from Gaussian FARIMA$(0,H_0-\frac12,0)$
sequences generated by circulant embedding, and $\mathcal J^1_J(t)$ is an
index set defined in~\cite{AyacheNum2025}. The series converges to
$Z^{(q,H)}$ almost surely and uniformly on compact
intervals~\cite{AyacheNum2025}. The known rate,
$O(J^{q/2}\,2^{-J(H-1/2)})$, involves an unknown random constant and
gives no quantitative guarantee at a fixed level $J$, so we check the
output directly.

We take $T=1$, $H=0.70$, $J=17$ (grid step $2^{-17}$), and the values
$0.99$ and $10^{-3}$ for the two tuning parameters of the index sets
of~\cite{AyacheNum2025}. We use the authors' code with three changes.
First, the circulant embedding method produces exactly centred FARIMA
sequences, and we use them without subtracting their empirical mean:
since a sequence covers the whole simulation horizon, this subtraction
forces the sum of all increments to vanish, so that $Z_1\approx0$ for
$q=1$, and it biases the products used for $q=2,3$. Second, the original
code divides each simulated path by the empirical standard deviation of
its own finest-scale increments; this random normalisation changes the
law of the process, in particular its variance and its tails, and we
remove it. Instead, all the paths of a given order are divided by one
deterministic constant $c_q$, computed from $100$ independent pilot paths
so that the increments over the lag $\Delta=2^{-11}$ have second moment
$\Delta^{2H}$; by self-similarity this corresponds to
$\E[(Z_1^{(q,H)})^2]=1$. Third, the random generator is seeded, so that
all results are reproducible. The scripts are available from the author
and will be deposited in a public repository.

Two checks support this normalisation. For each order, the mean square
of the increments of the unscaled series over the lag $2^{e-17}$, divided
by $2^{2H(e-17)}$, does not depend on $e$ up to sampling error for
$8\le e\le14$ (within $1\%$ for $q=1$ already from $e=4$). For $q=2,3$ it
is smaller at the finest lags: the truncated series lacks part of the
small-scale variation. At the lag $e=6$ used for $c_q$ it is $3.5\%$
($q=2$) and $2.2\%$ ($q=3$) below its value on the plateau, so that
$\E[(Z_1^{(q,H)})^2]$ is overestimated by about these amounts. Second,
with the main sample of $1000$ paths per order, the estimates of
$\E[(Z_1^{(q,H)})^2]$ are $0.91$, $0.98$ and $1.24$ for $q=1,2,3$, with
standard errors $0.04$, $0.08$ and $0.23$, and an independent sample of
$3000$ paths for $q=1$ gives $1.04$ with standard error $0.03$. These
values are compatible with $1$. The solution of~\eqref{eq:main_sde} is
approximated by the Euler scheme on the simulation grid. All results
below are based on $1000$ independent paths per order. For $q=3$ the
series converges more slowly, and the results should be read as
qualitative.

\subsection{Sample paths}

Figure~\ref{fig:hermite_paths} shows paths of $Z^{(q,H)}$ on $[0,10]$ for
$q\in\{1,2,3\}$ and $H\in\{0.6,0.7,0.8,0.9\}$. They are obtained from
paths on $[0,1]$ by self-similarity: $(Z^{(q,H)}_{10s})_{s\in[0,1]}$ has
the law of $(10^HZ^{(q,H)}_s)_{s\in[0,1]}$, so that the grid step on
$[0,10]$ is $10\cdot2^{-17}$. For each $H$, the scale constant is computed
as in Section~\ref{subsec:sim_method} from $40$ pilot paths, and, for a
given $q$, the four displayed paths are built from the same random input
(the same seed), so that the effect of $H$ can be seen directly: the
paths become smoother as $H$ increases. The three processes with the same
$H$ have the same covariance but different marginal laws.
Figure~\ref{fig:model_paths} shows the corresponding paths
of~\eqref{eq:main_sde} on $[0,10]$, with $a=0.40$, $b_0=3\%$,
$b_1=0.5\%$, $T_{\mathrm{per}}=0.25$ and $x_0=3\%$, for $\sigma_1=0$,
$\sigma_0=1\%$ (top) and for $\sigma_0=0.25\%$, $\sigma_1=5\%$,
$\eps=10^{-4}$ (bottom). Near the level $3\%$, the volatility in the
second case is $\sigma_0+\sigma_1\sqrt{0.03}\approx1.1\%$, close to that
of the first case, which explains why the two rows are similar. The
other experiments below use $H=0.70$ and the horizon $T=1$.

\begin{figure}[htbp]
\centering
\includegraphics[width=\linewidth]{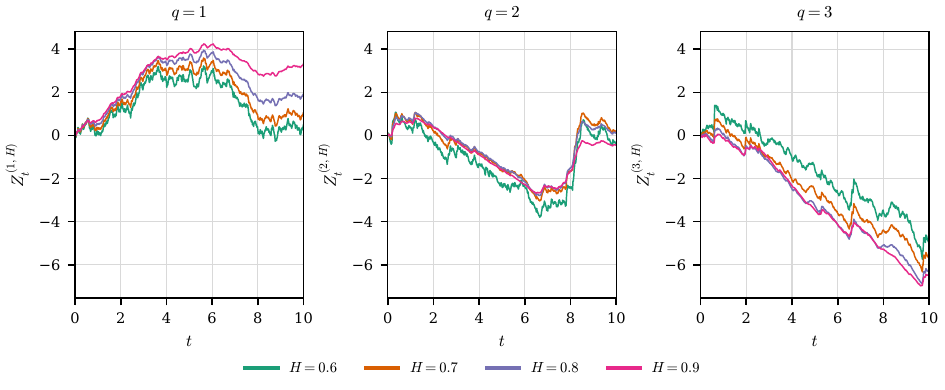}
\caption{Paths of $Z^{(q,H)}$ on $[0,10]$ for $q=1$ (left), $q=2$
(centre) and $q=3$ (right), with $J=17$. The colour gives the value of
$H$, the same in the three panels: green $H=0.6$, orange $H=0.7$,
purple $H=0.8$, pink $H=0.9$. In each panel, the four paths are built
from the same random input.}
\label{fig:hermite_paths}
\end{figure}

\begin{figure}[htbp]
\centering
\includegraphics[width=\linewidth]{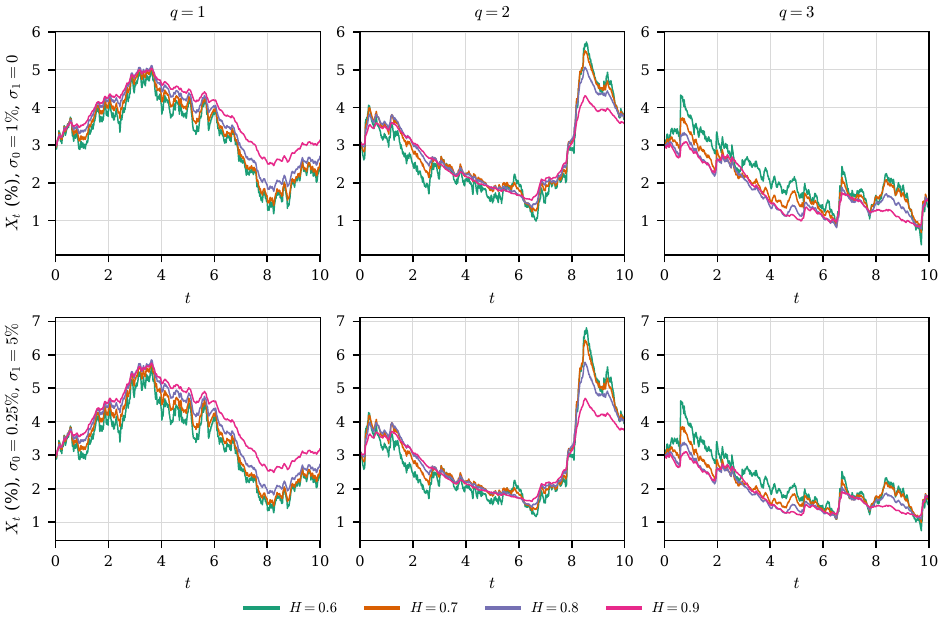}
\caption{Paths of~\eqref{eq:main_sde} on $[0,10]$ driven by the paths of
Figure~\ref{fig:hermite_paths}, for $q=1$ (left), $q=2$ (centre) and
$q=3$ (right), with $a=0.40$, $b_0=3\%$, $b_1=0.5\%$,
$T_{\mathrm{per}}=0.25$, $x_0=3\%$. Top: $\sigma_0=1\%$, $\sigma_1=0$.
Bottom: $\sigma_0=0.25\%$, $\sigma_1=5\%$, $\eps=10^{-4}$. The colour
gives the value of $H$, as in Figure~\ref{fig:hermite_paths}: green
$H=0.6$, orange $H=0.7$, purple $H=0.8$, pink $H=0.9$.}
\label{fig:model_paths}
\end{figure}

\subsection{Positivity}

With $H=0.70$, $T=1$ and the drift parameters of Figure~\ref{fig:model_paths}, we estimate
the probability that the Euler approximation of $X$ stays positive on
$[0,1]$, first for $\sigma_1=0$ and $\sigma_0\in\{0.5,1,2,4,8\}\%$, then
for $\sigma_0=0.25\%$, $\eps=10^{-4}$ and $\sigma_1\in\{1,2,5,10,20\}\%$.
We compare it with two quantities.
\begin{itemize}
\item The probability $\PP(C_1<\text{threshold})$, with the threshold
$m_0/(\sigma_0(2-e^{-a}))$ of~\eqref{eq:event_inclusion} for $\sigma_1=0$
and $\Lambda(m_0)/2$ of~\eqref{eq:general_inclusion} for $\sigma_1>0$,
where $\Lambda(m_0)$ is computed with the closed form of
Remark~\ref{rem:feller_ratio}. By these inclusions it is a lower bound
for the probability of staying positive. On each of the $3\times1000$
simulated paths and for each parameter value, the Euler path stays
positive whenever $C_1$ is below the threshold, in agreement with the
pathwise inclusions.
\item The bound of Theorems~\ref{thm:positivity}(i)
and~\ref{thm:positivity_general}(i), in which we take the best $p$ in
$\{1,\dots,12\}$ and replace $\E[C_1^p]$ by its Monte Carlo estimate
(negative values are replaced by $0$). This is an estimate of the bound,
not a rigorous bound.
\end{itemize}
The results are shown in Figure~\ref{fig:positivity}, with $95\%$
confidence intervals (Wilson) for the frequencies.

The frequencies show that positivity is not automatic: for $\sigma_1=0$
and $\sigma_0=8\%$, only about half of the paths stay positive
($0.51$, $0.42$ and $0.54$ for $q=1,2,3$). For $\sigma_1=0$ and
$\sigma_0=2\%$, all $1000$ paths stay positive for $q=2$, against $966$
for $q=1$ and $965$ for $q=3$. This is consistent with the shape of the
Rosenblatt law, which is positively skewed with a light left tail: large
decreases of $Z^{(2,H)}$, which push $X$ towards $0$, are less likely
than for the Gaussian case. For $\sigma_1>0$ and $\sigma_0=0.25\%$, the
paths stay positive in more than $96\%$ of the cases for all parameters,
because the volatility decreases with the level.

The lower bound $\PP(C_1<\text{threshold})$ is far below the
frequencies when the volatility is large. For $\sigma_1=0$ and
$\sigma_0=4\%$, it is between $0.16$ and $0.22$, while the frequencies
are between $0.75$ and $0.78$. This gap is the price of discarding the
drift, which pulls $X$ back towards the target
(Remark~\ref{rem:feller_ratio}). The Markov-type bound is smaller still,
and it is informative only for the smallest volatilities. The
exponential bounds~\eqref{eq:positivity_exp} and~\eqref{eq:general_exp}
are trivial for all the parameters considered here: with the estimated
values $\kappa\approx1.08$, $1.10$ and $1.22$ for $q=1,2,3$, the term
$e^{2}\exp(\cdots)$ exceeds $1$. Their interest lies in the form of the
dependence on $m_0$, $\sigma_0$, $\sigma_1$ and $q$, not in their
numerical values.

\begin{figure}[htbp]
\centering
\includegraphics[width=\linewidth]{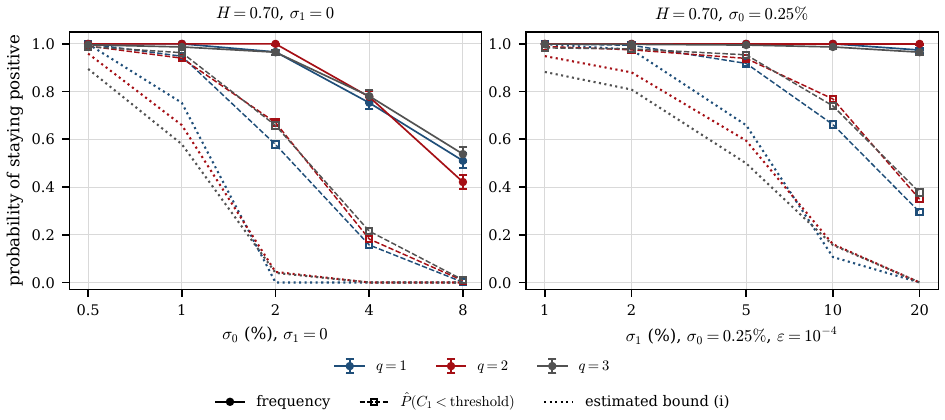}
\caption{Probability that $X$ stays positive on $[0,1]$, for $H=0.70$: frequency over
$1000$ paths with $95\%$ confidence intervals (solid lines), estimate of
$\PP(C_1<\text{threshold})$ (dashed lines) and estimated bound (i) of
Theorems~\ref{thm:positivity} and~\ref{thm:positivity_general} (dotted
lines). The colour gives the order: blue $q=1$, red $q=2$, grey
$q=3$. Left:
$\sigma_1=0$, as a function of $\sigma_0$.
Right: $\sigma_0=0.25\%$, $\eps=10^{-4}$, as a function of $\sigma_1$.
Other parameters as in Figure~\ref{fig:model_paths}.}
\label{fig:positivity}
\end{figure}

\subsection{Marginal distributions}

By self-similarity and stationarity of the increments, every increment
$Z^{(q,H)}_{t+\Delta}-Z^{(q,H)}_t$ has the law of $\Delta^HZ_1^{(q,H)}$.
For $q=2$, $Z_1^{(2,H)}=\sum_k\mu_k(\xi_k^2-1)$, where the $\xi_k$ are
independent standard Gaussian variables and the $\mu_k$ are the
eigenvalues of the operator with kernel proportional to $|u-v|^{H-1}$ on
$[0,1]^2$~\cite{tudorbook,VeilletteTaqqu2013}. Its skewness and excess
kurtosis are $2\sqrt2\sum_k\mu_k^3/(\sum_k\mu_k^2)^{3/2}$ and
$12\sum_k\mu_k^4/(\sum_k\mu_k^2)^2$. A Galerkin discretisation of the
operator with $n\le4000$ piecewise constant functions, extrapolated in
$n$, gives approximately $2.07$ and $7.63$ for $H=0.70$. The same
eigenvalues give the density of the standardised variable, by numerical
inversion of the characteristic function
$\theta\mapsto\prod_k(1-2i\theta\mu_k)^{-1/2}e^{-i\theta\mu_k}$, where
the contribution of the eigenvalues omitted by the discretisation is
replaced by a Gaussian term with the same variance.

Figure~\ref{fig:distributions} shows the empirical distribution of the
standardised increments over the lag $\Delta=2^{-11}$, pooled over $40$
paths per order ($81\,880$ increments). For $q=1$ the skewness and
excess kurtosis are $0.02$ and $0.04$. For $q=2$ they are $2.24$ and
$8.3$, reasonably close to the values above: the long-range dependence
of the increments makes the effective sample size much smaller than the
number of increments. The histogram agrees with the density of the
standardised Rosenblatt variable down to densities of order $10^{-4}$
(Figure~\ref{fig:distributions}, centre). The histogram stops near $-1.5$ on the left: the
Rosenblatt law has a very light left tail, since each term
$\mu_k(\xi_k^2-1)$ is bounded below by $-\mu_k$, and its standardised
density is already below $10^{-6}$ to the left of $-2.2$, a level that
cannot be observed with this number of increments. For $q=3$ the law should be
symmetric, and the excess kurtosis is large ($58.5$); the empirical
skewness, $-1.05$, is driven by a few extreme increments, and these
empirical moments are not reliable estimates. We use this figure only to
illustrate the qualitative shape of the marginal laws: Gaussian for
$q=1$, skewed with a heavy right tail for $q=2$, symmetric with heavy
tails for $q=3$.

\begin{figure}[htbp]
\centering
\includegraphics[width=\linewidth]{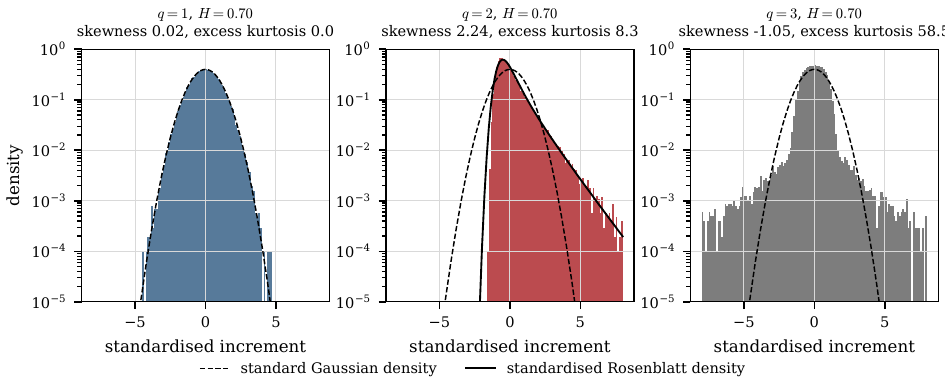}
\caption{Empirical distributions (log scale) of the standardised
increments of $Z^{(q,H)}$ over the lag $\Delta=2^{-11}$, for $q=1$ (left,
blue), $q=2$ (centre, red) and $q=3$ (right, grey), with $H=0.70$,
$J=17$ and $40$ paths per order. Dashed curve: standard Gaussian
density. Solid curve (centre): density of the standardised Rosenblatt
variable $Z^{(2,H)}_1$, computed from the eigenvalues of
Section~\ref{sec:simulations}.}
\label{fig:distributions}
\end{figure}

\section*{Acknowledgements}

The author thanks Laurent Loosveldt, Yassine Nachit and Ivan Nourdin for
valuable discussions on~\cite{loosveldt2025}, and Antoine Ayache, Julien
Hamonier and Laurent Loosveldt for providing their simulation
code~\cite{AyacheNum2025}. This research was conducted at the Toulouse
School of Economics, Universit\'e Toulouse Capitole. The author
acknowledges the support of the French National Research Agency (ANR)
under grant ANR-17-EURE-0010 (Investissements d'Avenir program).

\end{document}